\documentclass[12pt]{amsart}
\usepackage{amsmath,amssymb,amsthm}

\usepackage{ifthen,verbatim}
\usepackage{mathrsfs}
\usepackage{color}
\usepackage{url}
\usepackage{pdfpages}
\usepackage[english]{babel}
\usepackage{verbatim,here}
\usepackage[T1]{fontenc}
\usepackage{floatflt,graphicx,graphics}
\usepackage{a4wide}
\graphicspath{{figures}}
\usepackage[autolinebreaks]{mcode}

\numberwithin{equation}{section}
\nonstopmode
\theoremstyle{plain}

\newtheorem{theorem}[equation]{Theorem}

\newtheorem{lem}[equation]{Lemma}
\newtheorem{lemma}[equation]{Lemma}

\newtheorem{proposition}[equation]{Proposition}

\theoremstyle{definition}
\newtheorem{rem}[equation]{Remark}
\newtheorem{remark}[equation]{Remark}

\newtheorem{nonsec}[equation]{}

{\qed\bigskip}

\newcounter{alphabet}

\newcommand{\be}{\begin{eqnarray}}
\newcommand{\ee}{\end{eqnarray}}
\newcommand{\ba}{\begin{array}}
\newcommand{\ea}{\end{array}}
\newcommand{\ben}{\begin{eqnarray*}}
\newcommand{\een}{\end{eqnarray*}}

\newcommand{\C}{{\mathbb C}}
\newcommand{\R}{{\mathbb R}}

\newcommand{\B}{\mathbb{B}}

\newcommand{\capa}{\mathrm{cap}\,}
\renewcommand{\th}{\,\textnormal{th}}

\newcommand{{\tth}}{\mathrm{th}}
\newcommand{{\sh}}{\mathrm{sh}}
\newcommand{{\ch}}{\mathrm{ch}}

\newcommand{\HH}{{\mathbb H}^2}

\newcommand{\K}{\mathcal{K}}

\newcommand{\arsh}{\,\mathrm{arsh}\,}
\newcommand{\arth}{\,\mathrm{arth}\,}
\newcommand{\tnh}{\mathrm{th}}

\newcommand{\sn}{\mathrm{sn}}

\renewcommand{\Im}{{\,\operatorname{Im}\,}}

\newcommand {\mM} {\mathsf{M}}
\newcommand {\M} {\mathsf{M}}

\newcommand{\QM}{\mathop{\mathrm{QM}}}
\newcommand{\QMt}{\mathop{\mathrm{QMt}}}

\renewcommand{\i}{\mathrm{i}}
\newcommand{\bs}{{\bf s}}

\font\fFt=eusm10 
\font\fFa=eusm7  
\font\fFp=eusm5  
\def\K{\mathchoice
{\hbox{\,\fFt K}}
{\hbox{\,\fFt K}}
{\hbox{\,\fFa K}}
{\hbox{\,\fFp K}}}

\newcounter{minutes}
\divide\time by 60
\newcounter{hours}
\multiply\time by 60
\addtocounter{minutes}{-\time}

\begin{document}

\bibliographystyle{amsplain}
\title
{
Condenser capacity and hyperbolic perimeter
}

\def\thefootnote{}
\footnotetext{
\texttt{\tiny File:~\jobname .tex,
          printed: \number\year-\number\month-\number\day,
          \thehours.\ifnum\theminutes<10{0}\fi\theminutes}
}
\makeatletter\def\thefootnote{\@arabic\c@footnote}\makeatother

\author[M.M.S. Nasser]{Mohamed M. S. Nasser}
\address{Department of Mathematics, Statistics and Physics, Qatar University, P.O.Box 2713, Doha, Qatar}
\email{mms.nasser@qu.edu.qa}
\author[O. Rainio]{Oona Rainio}
\address{Department of Mathematics and Statistics, University of Turku, FI-20014 Turku, Finland}
\email{ormrai@utu.fi}
\author[M. Vuorinen]{Matti Vuorinen}
\address{Department of Mathematics and Statistics, University of Turku, FI-20014 Turku, Finland}
\email{vuorinen@utu.fi}

\begin{abstract}
We study the conformal capacity by using novel computational algorithms based on implementations of the fast multipole method, and analytic techniques. Especially, we apply domain functionals to study the capacities of condensers $(G,E)$ where $G$ is a simply connected domain in the complex plane and $E$ is a compact subset of $G$. Due to conformal invariance, our main tools are the hyperbolic geometry and functionals such as the hyperbolic perimeter of $E$. Our computational experiments demonstrate, for instance, sharpness of established inequalities. In the case of model problems with known analytic solutions, very high precision of computation is observed.
\end{abstract}
\maketitle

\section{Introduction}

A \emph{condenser} is a pair $(G,E)$, where $G\subset\R^n$ is a domain and $E$ is a compact non-empty subset of $G$. The \emph{ conformal capacity} of this condenser is defined as  \cite{GMP,HKV,hkm}
\begin{align}\label{def_condensercap}
{\rm cap}(G,E)=\inf_{u\in A}\int_{G}|\nabla u|^n dm,
\end{align}
where $A$ is the class of $C^\infty_0(G)$ functions $u: G\to[0,\infty)$ with $u(x) \ge 1$ for all $x \in E$ and $dm$ is the $n$-dimensional Lebesgue measure. The conformal capacity of a condenser is one of the key notions in potential theory of elliptic partial differential equations \cite{hkm, m} and it has numerous applications to geometric function theory, both in the plane and in higher dimensions, \cite{du, GMP, GoRe, HKV}. 

The {\it isoperimetric problem} is to determine a plane figure of the largest possible area whose boundary has a specified length, or, perimeter. Constrained extremal problems
of this type, where the constraints involve geometric or physical quantities,  have been studied in several thousands of papers, see e.g. \cite[p. 151]{bae}, \cite{bhs,bds,fu, m03}.
Motivated by the fact that many  \emph{domain functionals} of mathematical physics such as capacity, moment of inertia, principal frequency, or torsional rigidity have analytic formulas only in rare exceptional
cases,  P\'olya and Szeg\H o developed \cite{ps} a systematic theory to prove upper and 
lower bounds for these functionals in terms of simpler domain functionals such as area, perimeter, inradius, and circumradius. In addition to these domain functionals, they used  the method of \emph{symmetrization} 
as a method to transform a condenser $(G,E)$ onto another, symmetrized condenser 
$(G^*,E^*)\,.$ The key fact here is that the integral in \eqref{def_condensercap} decreases
under symmetrization  \cite[p. 96]{bae} and hence we will have a lower bound
\begin{align}\label{ine_symmetrization}
\capa(G,E)\geq\capa(G^*,E^*).
\end{align}
There are several variants of the symmetrization method and  depending on which  one is applied, the sets $G^*$ and $E^*$ exhibit various symmetry properties with respect to spheres or hyperplanes \cite[p. 253]{bae}. Due to this symmetry, the capacity of the symmetrized condenser is often easier to estimate than the original one.

Note that, while the lower bound of \eqref{ine_symmetrization} is clearly sharp if $(G,E)=(G^*,E^*)$, in most cases the symmetrization method only yields crude estimates. 
The {domain functionals} of  P\'olya and Szeg\H o~\cite{ps}, such as volume, area, perimeter, inradius, and circumradius, expressed in terms of Euclidean geometry, have numerous applications and they behave well
under symmetrization, but they
do not seem to be natural in the study of conformal capacity. The reason is that Euclidean
geometry does not reflect optimally the conformal invariance of the conformal capacity.
This observation led us to  use hyperbolic geometry, which is available in the planar case $n=2$, when the domain $G$ is the unit disk $\B^2$ or, more generally by Riemann's mapping theorem, a simply connected plane domain. For dimensions $n \ge 3$, Liouville's theorem
states that conformal mappings are M\"obius transformations and hence
Riemann's theorem does not apply. For generalized versions of Liouville's theorem, see  Yu.G. Reshetnyak \cite[Thm 5.10, p. 171, Thm 12.4, p. 251]{res} and F.W. Gehring et al. \cite[Thm 6.8.4, p. 336]{GMP}.

Many authors have  proved upper and lower bounds for several kinds of capacities, including the conformal capacity that we are focusing here on, see for instance V. Maz\'{}ya \cite{m}. In spite of all this work, there does not seem to exist bounds in the form
\begin{align}
L\leq\capa(G,E)\leq U    
\end{align}
with a quantitative upper bound for the deviation $U-L$, even in the simplest case $G=\B^2$. In particular, there is no quantitative two-sided variant for the symmetrization inequality \eqref{ine_symmetrization}. Here, a fundamental difficulty is that the value of $\capa(G,E)$ is unknown.
For the isoperimetric inequality, quantitative variants have been proved by N. Fusco \cite{fu} and, in the framework of the hyperbolic geometry,  by  V. B\"ogelein, F. Duzaar, Ch. Scheven \cite{bds}. Inequalities for the $p$-capacity were proved very recently by J. Xiao \cite{x}
and E. Mukoseeva \cite{muk}. In a series of papers~\cite{BSV,BBGGHV,HNV,HRV, HRV3}, the third author with several coauthors has studied numerical computation of condenser capacities using the finite element method.

Despite the extensive literature dealing with condenser capacity~\cite{GoRe,HKV,m}, the actual values of $\capa(G,E)$ have remained rather elusive quantities. This is largely due to the unavailability of computational tools that can be used for wide ranges of domains. In fact, we have not seen a systematic compilation of concrete bounds for the capacity published since the pioneering work of P\'olya and Szeg\H o \cite{ps}.

In this paper, our goal is to combine analytic methods with efficient numerical techniques and with extensive experiments
to demonstrate the precision of the methods and the behavior of the numerical algorithms. To find new upper and lower bounds for the condenser capacity, we  introduce {\it new kinds of domain functionals} expressed in terms of hyperbolic geometry of the unit disk $\B^2\,.$ 
The numerical methods of these experiments are based on  {the boundary integral method developed by the first author and his coauthors in } a series of recent papers \cite{LSN17,
Nas-ETNA,Nas-cmft15,nv1,nvs,nv}.

The first question we study is whether the symmetrization lower bound \eqref{ine_symmetrization} for $\capa(\B^2,E)$, where
the interior of $E$ is a simply connected domain with a piecewise smooth boundary $\partial E$, could be improved by using the hyperbolic perimeter of $E$. We will give  examples to demonstrate that this is not true in general. For convex sets $E$, we have a positive result.
Our experiments led us to an experimental verification of the next two theorems. We learned
afterwards that both results are well-known.

\begin{theorem}  {\rm (R. K\"uhnau \cite[p. 99, Thm 9.8]{ku})}\label{KuhnauThm}
If $E\subset\B^2$ is a convex polygon, then
\begin{align*}
\capa(\B^2,E)\geq\capa(\B^2,I),
\end{align*}
where $I=[0,r]$ is a segment with the same hyperbolic perimeter as $\partial E$.
\end{theorem}

Note that a spherical symmetrization argument, see Lemma \ref{lem_capgamma}(1) below, shows that
\[
\capa(\B^2,I)\geq \capa(\B^2,[0,s]),
\]
where the segment $[0,s] \subset \B^2$ has the same {\it  hyperbolic diameter} as the set $E\,$
and therefore $s\leq r$ and hence Theorem \ref{KuhnauThm} gives a better lower bound than the
symmetrization method.
 
\begin{theorem} {\rm (F.W. Gehring  \cite[Corollary 6]{g})} \label{fwg71}
\ If $E\subset\B^2$ is a simply connected domain, then
\begin{align*}
\capa(\B^2,E)\leq \capa(\B^2,F),
\end{align*}
where $F$ is a disk with {hyperbolic} perimeter equal to that one of $E\,.$
\end{theorem}

For some observations about Theorem \ref{fwg71} see Remark \ref{fwgRmk}.

The remaining part of this paper is organized as follows:
Sections~\ref{sec:pre} and~\ref{sec:cap} present preliminary materials about hyperbolic geometry, quadrilateral and its modulus, and conformal capacity, which will be used in the following sections.
Analytical results for simple condensers are presented in Section~\ref{sec:comp} and numerical methods for computation of the capacity of condensers and the modulus of quadrilaterals are presented in Section~\ref{sec:num}. In these two sections, we  compare various lower bounds for the capacity to the symmetrization inequality~\eqref{ine_symmetrization}. We also give a lower bound for $\capa(G,E)$ in the case 
when $G \setminus E$ is a polygonal ring domain. This lower bound is sharp in the case when the polygonal ring domain has certain regularity properties. The results of this computational experiment are presented in the form of tables and graphics.
In the final section of this paper, Section~\ref{sec:epi}, we point out that finding connections between the geometric domain functional $d(E)/d(E,\partial G)$ and $\capa(G,E)$ seems to offer problems for further investigations.\\

\textbf{Acknowledgements.} The second author was financially supported by the University of Turku Graduate School UTUGS. We are indebted to D. Betsakos and R. K\"uhnau for informing us about the literature \cite[Corollary 6]{g} and \cite[p. 99, Thm 9.8]{ku}, resp., related to Theorems \ref{fwg71} and \ref{KuhnauThm}. We are also thankful to the referees for their useful and constructive comments.

\section{Preliminaries}\label{sec:pre}

Consider first the notations for the Euclidean metric. Let $d_G(x)$ be the Euclidean distance between a point $x$ in a domain $G$ and the boundary $\partial G$. Denote the Euclidean diameter of a nonempty set $F$ by $d(F)$ and the Euclidean distance between two non-empty sets $E,F$ by $d(E,F)$. Denote the Euclidean open ball with a center $x\in\R^n$ and a radius $r>0$ by $B^n(x,r)$, the corresponding closed ball by $\overline{B}^n(x,r)$ and its boundary sphere by $S^{n-1}(x,r)$. Suppose that $x=0$ and $r=1$ here, if they are not otherwise specified. 

\begin{nonsec} {\bf Hyperbolic geometry.}
Define now the hyperbolic metric in the Poincar\'e unit ball $\B^n$ as in  \cite{be}, \cite[(2.8) p. 15]{BM}
\begin{align}\label{myrho}
\text{sh}^2\frac{\rho_{\B^n}(x,y)}{2}&=\frac{|x-y|^2}{(1-|x|^2)(1-|y|^2)},\quad x,y\in\B^n.
\end{align}
Here and below, sh, ch and th stand for the hyperbolic sine, cosine and tangent functions, respectively. Let $J[x,y]$ be the hyperbolic segment between the points $x,y$ and $[x,y]$ its Euclidean counterpart. Note that, for any simply connected domain $G\subsetneq\R^2$, one can choose a conformal map $f\,:\,G\to \B^2=f(G)$ by means of Riemann's mapping theorem and thus define the hyperbolic metric $\rho_G$ in $G$ by \cite{BM}
\begin{equation}\label{eq:rhoG}
\rho_G(x,y) = \rho_{\B^2}(f(x),f(y)),\quad x,y\in G.
\end{equation}
\end{nonsec}

\begin{nonsec} {\bf Hyperbolic disks.} We use the notation
\[ B_\rho(x,M) = \{z \in \mathbb{B}^2: \rho_{\mathbb{B}^2}(x,z)<M\}\]
for the hyperbolic disk centered at $x\in \mathbb{B}^2$ with radius $M>0\,.$
It is a basic fact that they are Euclidean disks with the center and radius given
by \cite[p. 56, (4.20)]{HKV}
\begin{equation}\label{hypdisk}
 \begin{cases}
        B_\rho(x,M)=B^2(y,r)\;,&\\
  \noalign{\vskip5pt}
      {\displaystyle y=\frac{x(1-t^2)}{1-|x|^2t^2}\;,\;\;
        r=\frac{(1-|x|^2)t}{1-|x|^2t^2}\;,\;\;t={\th}\, ( M/2)\;.}&
\end{cases}
\end{equation}
Note the special case $x= 0$,
\begin{equation} \label{hypDat0}
 B_\rho(0,M)=B^2(0,{\th}\, ( M/2)) \,.
\end{equation}
\end{nonsec} 

\begin{lemma}\cite[Thm 7.2.2, p. 132]{be} \label{perimBe}
The area of a hyperbolic disc of radius $r$ is $4\pi{\rm sh}^2(r\slash2)$ and  the length (or the perimeter) of a hyperbolic circle of radius $r$ is $2\pi{\rm sh}(r)$.
\end{lemma}

\begin{nonsec}{\bf Quasihyperbolic metric.} For a domain $G\subsetneq\R^n$, define the weight function as in \cite[(5.1), p. 68]{HKV}
\begin{align*}
w:G\to (0,\infty),\quad w(x)=1\slash d_G(x),\quad x\in G.    
\end{align*}
By \cite[(5.2), p. 68]{HKV}, the \emph{quasihyperbolic distance} between $x,y\in G$ is now
\begin{align*}
k_G(x,y)=\inf_{\alpha\in\Gamma_{xy}}\int_\alpha w(x)|dx|,   
\end{align*}
where $\Gamma_{xy}$ is the family of all rectifiable curves in $G$ joining $x$ and $y$. Note that if
$G$ is a simply connected domain in the plane, then the quasihyperbolic metric fulfills the inequality
\begin{equation} \label{kvsrho}
\rho_{G}(x,y) \leq  2k_{G}(x,y)\leq  4\rho_{G}(x,y)     
\end{equation}
for all points $x,y\in G$ \cite[p.~21, (4.15)]{garmar}. 
\end{nonsec} 

\medskip

The next lemma is based on a standard covering lemma. Note that here the connectedness
of the set $F$ is essential as shown in \cite{vu81}. 

\begin{lemma}\cite[Lemma 2.17, p. 347 \& Cor. 2.18, p. 348]{vu81} \label{vu81}
For $n\geq2$, there exists a constant $c(n)>0$ such that, for a domain $G\subset\R^n,$ every continuum $F\subset G$ can be covered by a family of balls 
$$ \{  B^n(z_j, \lambda d_G(z_j): j =1,\dots,m\} , \quad m\le c(n)\left(2+\frac{d(F)}{\lambda d(F,\partial G)}\right)^n\,.$$
In particular, $ k_G(B^n(z_j, \lambda d_G(z_j)))\leq 2 \log(1/(1-\lambda))$ for all $j$ and 
$k_G(F)\leq  2 m \log(1/(1-\lambda))\,.$
\end{lemma}

\begin{nonsec}{\bf Quadrilateral and its modulus.} \label{quadrDef}
A bounded Jordan curve in the complex plane divides the extended complex plane $\C_\infty=\C\cup\{\infty\}$ into two domains $D_1$ and $D_2$ so that the common boundary of these domains is the curve in question. One of these domains is bounded and the other one is unbounded. If $D_1$ is the bounded domain and $z_1,z_2,z_3,z_4\in\partial D_1$ are distinct points occurring in this order when traversing $\partial D_1$ in the positive direction, then 
$(D_1;z_1,z_2,z_3,z_4)$ is a \emph{quadrilateral} \cite{du}.

By Riemann's mapping theorem, the domain $D_1$ can be now be mapped onto a rectangle $(0,1)\times(0,h)$ by a conformal mapping $f$ such that $f(z_1)=0$, $f(z_2)=1$, $f(z_3)=1+hi$ and $f(z_4)=hi$. The unique number $h$ here is called the \emph{conformal modulus of the quadrilateral} $(D_1;z_1,z_2,z_3,z_4)\,,$ denoted by ${\rm mod} (D_1;z_1,z_2,z_3,z_4) \,.$ It follows
from the definition that
\begin{equation} \label{recipIdty}
{\rm mod} (D_1;z_1,z_2,z_3,z_4) = 1/{\rm mod} (D_1;z_2,z_3,z_4, z_1) \,. 
\end{equation}
\end{nonsec}

\section{Capacity}\label{sec:cap}
Let $G\subset\R^n$ be a domain and $E\subset G$ a compact non-empty set. For $k=1,2,\ldots$, choose domains $G_k$ and compact sets $E_k$ such that 
\[
E\subset E_{k+1}\subset E_k\subset G_k\subset G_{k+1}\subset G, \quad G=\bigcup G_k \quad{\rm and}  \quad E=\bigcap E_k.
\] 
Then it is well-known that
\begin{align}\label{cap_convergence}
\lim_{k\to\infty}\capa(G_k,E_k)=\capa(G,E),  
\end{align}
see \cite[p. 167]{GMP}. Unfortunately, there does not seem to exist a quantitative estimate for the speed of convergence in \eqref{cap_convergence}.

Numerous variants of the definition \eqref{def_condensercap} of capacity are given in \cite{GMP}. First, the family $A$ may be replaced by several other families by \cite[Lemma 5.21, p. 161]{GMP}. Furthermore,
\begin{align}
\capa(G,E)=\M(\Delta(E,\partial G;G)),    
\end{align}
where $\Delta(E,F;G)$ is the family of all curves joining nonempty sets $E$ and $F$ in the closure of the domain $G$ and $\M$ stands for the modulus of a curve family \cite[Thm 5.23, p. 164]{GMP}. For the basic facts about capacities and moduli, the reader is referred to \cite{GMP, HKV, v1}.

\begin{lemma}\label{cgqm_5.14}\cite[(7.3), p. 107]{HKV}\newline
(1) If $0<a<b$ and $D=\overline{B}^n(b)\backslash B^n(a)$,
\begin{align*}
\M(\Delta(S^{n-1}(a),S^{n-1}(b);D))=\omega_{n-1}(\log({b}/{a}))^{1-n}.
\end{align*}
(2) If $R>0$, then for $x \in \B^n$, 
\begin{align*}
\M(\Delta(S^{n-1},B_{\rho}(x,R);\B^n))=\omega_{n-1}(\log({1}/{{\rm th} (R/2)}))^{1-n} \,.
\end{align*}
Here, $\omega_{n-1}$ is the $(n-1)$-dimensional surface area of the unit sphere $S^{n-1}.$ In particular, $\omega_1=2\pi.$
\end{lemma}

\begin{proposition} \label{capPerim} If $E\subset \B^2$ is a disk with hyperbolic perimeter equal to $P>0,$ then
\begin{equation}\label{eq:disk-cap-P}
\capa(\B^2,E) = 2 \pi/\log \frac{2\pi + \sqrt{P^2 + (2\pi)^2}}{P}\,.
\end{equation}
\end{proposition}

\begin{proof} It readily follows from Lemma \ref{perimBe} that $E$ is a hyperbolic disk with the
hyperbolic radius $r={\rm arsh}(P/(2\pi))\,.$ Next, it follows from Lemma \ref{cgqm_5.14} (2) that
\[
\capa(\B^2,E) = 2 \pi/ \log(1/{\rm th}(r/2))\,.
\]
The desired formula now follows using the half angle formula for ${\rm th}\,.$
\end{proof}

\begin{remark}\label{fwgRmk} The proof of Gehring's theorem \ref{fwg71} follows from
 \cite[Cor. 6]{g} and Proposition
\ref{capPerim}. Observe that the hyperbolic metric used here is twice that of~\cite{g}
and therefore in Proposition~\ref{capPerim} we have the constant $2\pi$ (in the argument of the log function) whereas Gehring has $\pi$ in~\cite[Cor. 6]{g}.
\end{remark}

\begin{nonsec} {\bf The Gr\"otzsch and Teichm\"uller capacities.}
The following decreasing homeomorphisms are called {\it the Gr\"otzsch and Teichm\"uller capacities}, respectively \cite[(7.17), p. 121]{HKV}:
\begin{align*}
&\gamma_n:(1,\infty)\to (0,\infty),\quad \gamma_n(s)=\M(\Delta(\overline{\B}^n,[se_1,\infty];\R^n)),\quad s>1,\\
&\tau_n:(0,\infty)\to (0,\infty),\quad\tau_n(s)=\M(\Delta([-e_1,0],[se_1,\infty];\R^n)),\quad s>0.
\end{align*}
Here, $e_1,\ldots,e_n$ are the unit vectors of $\R^n$. These capacity functions fulfill $\gamma_n(s) = 2^{n-1} \tau_n(s^2-1)$ for $s>1$ and several estimates are given in \cite[Chapter 9]{HKV} for $n\ge3\,.$ In the case $n=2, r\in(0,1),$ the following explicit
formulas are well-known \cite[(7.18), p. 122]{HKV}, 
\begin{equation} \label{capGro}
\gamma_2(1/r)=\frac{2\pi}{\mu(r)}\,; \quad \mu(r)=\frac{\pi}{2}\frac{\K(r')}{\K(r)},\quad
\K(r)=\int^1_0 \frac{dx}{\sqrt{(1-x^2)(1-r^2x^2)}},\quad r'=\sqrt{1-r^2}\,.
\end{equation}
\end{nonsec}

\begin{nonsec}{\bf Quadrilateral modulus and curve families.} \label{quadrCurves}
The modulus of a quadrilateral $(D;z_1,z_2,z_3,z_4)$ defined in  \ref{quadrDef} is connected
with the modulus of the family of all curves in $D,$ joining the opposite boundary arcs $(z_2,z_3)$ and $(z_4,z_1),$ in a very simple way, as follows
\begin{equation} \label{2moduli}
{\rm mod}(D;z_1,z_2,z_3,z_4) = \M(\Delta((z_2,z_3), (z_4,z_1);D)) \,.
\end{equation}
\end{nonsec}

The next lemma is based on the symmetrization method. 

\begin{lemma}\label{lem_capgamma}\cite[Lemma 9.20, p. 163]{HKV}\newline
(1) If $x,y\in\B^n, x\neq y\,,$ and $E\subset\B^n$ is a continuum with $x,y\in E$, then 
\begin{align*}
\capa(\B^n,E)
\geq
\gamma_n\left(\frac{1}{\tnh(\rho_{\B^n}(x,y)\slash2)}\right).
\end{align*}
Equality holds here if $E$ is the geodesic segment $J[x,y]$ of the hyperbolic
metric joining $x$ and $y\,.$

(2) If $G$ is a simply connected domain in $\R^2\,,$ $E\subset G$ is a continuum, and
$x,y \in G, x\neq y\,,$ then
\begin{align*}
\capa(G,E)
\geq
\gamma_2\left(\frac{1}{\tnh(\rho_{G}(x,y)\slash2)}\right).
\end{align*}
\end{lemma}

\begin{nonsec}{\bf Jung radius in quasihyperbolic geometry.} \label{jung}
For a domain $G\subset {\mathbb{R}^n}$ and a compact non-empty set $E\subset G$, define {\it the Jung radius} of $E$ in the quasihyperbolic metric as
\begin{align}
r_{k-Jung}(E)=\inf\{t>0\;:\; E\subset B_k(z,t) \,\, {\rm for \, some }\, z \in G\},    
\end{align}
where $B_k(z,t)$ is the quasihyperbolic ball centered at a point $z\in G$ with radius $t>0$. Because
$(G,k_G)$ is a metric space, it is clear that
\begin{align}
r_{k-Jung}(E)\leq k_G(E).    
\end{align}
From the monotonicity property of the capacity, we immediately get the upper bound
\begin{align} \label{capJung}
\capa(G,E)\leq\capa(G,B_k(z,T))    
\end{align}
for some $z\in G$ and $T=r_{k-Jung}(E)\le k_G(E)$. In particular, by Lemmas \ref{vu81} and \ref{cgqm_5.14} and the subadditivity of the modulus \cite[Ch. 7]{HKV}, for $\lambda \in (0,1),$
\begin{align} \label{capvu81}
\capa(G,E)\leq 2 m\, \omega_{n-1}(\log(1/\lambda))^{1-n}\,
\end{align}
for a continuum $E \subset G$ where $m$ is as in Lemma \ref{vu81}.
\end{nonsec}


\section{Analytical results for simple condensers}\label{sec:comp}
In this section, we study $\capa(G,E)$ and relate its values to various domain functionals. 
In particular, we focus on the symmetric condenser of Lemma \ref{Lem:LB-Tri-h} and show
that the capacity  $\capa(G,E)$ cannot be estimated from below in the same way as in Theorem~\ref{KuhnauThm} because here $E$ is nonconvex. We also 
consider the case when $G \setminus E$ is a polygonal ring domain defined in Subsection~\ref{polyringdomain}. 
We apply the
Schwarz-Christoffel transformation to give
an algorithm for a lower bound  of its capacity. Note that here it is not required that
$E$ is convex. This algorithm will be implemented
 in the next section.

\begin{nonsec}{\bf Symmetric segments.} We consider here condensers of the form
$(\mathbb{B}^2,E)$ where $E= \cup_{k=1}^m [0, s e^{i 2 \pi k/m}]$ and $m=3,4,5,\ldots,$ 
$0<s<1\,.$ For these condensers the capacity can be explicitly given.
\end{nonsec}

\begin{lem}\label{Lem:LB-Tri-h} 
The capacity $\capa( \mathbb{B}^2,E)$ is given by
\[
\capa( \mathbb{B}^2,E) = \frac{2 m \pi}{\mu\left(s^m\right)}.
\]
\end{lem}
\begin{proof}
Let $\theta= 2\pi/m$ and
\[
\hat{D}_1=\left\{z\,:\, z\in  \mathbb{B}^2, \quad 0<\arg z<\theta\right\}.
\]
 The domain $\hat{D}_1$ can be mapped by the conformal mapping 
\[
g(z)= z^{m/2}
\]
onto the upper half of the unit disk and the two segments from $s$ to $0$ and from $0$ to $s\,e^{\theta\i}$ are mapped onto the segment $[-s^{m/2},s^{m/2}]$. 
Let $\tilde\Delta$ be the family of curves in the upper half of the unit disk connecting the segment $[-s^{m/2},s^{m/2}]$ to the upper half of the unit circle. 
Then by symmetry,
\begin{equation}\label{eq:hT-1}
\capa( \mathbb{B}^2,E)=\mM(\Delta(E,S^1; \mathbb{B}^2))=m\mM(\tilde{\Delta}).
\end{equation}

By symmetry, it also follows from~\cite[7.12,~9.20]{HKV},
\begin{equation*} 
\mM(\tilde\Delta)=\frac{1}{2}\gamma_2\left(\frac{1}{\th\frac{1}{2}\rho_{ \mathbb{B}^2}(-s^{m/2},s^{m/2})}\right)= \frac{\pi}{\mu\left(\th\frac{1}{2}\rho_{ \mathbb{B}^2}(-s^{m/2},s^{m/2})\right)}.
\end{equation*}
Using the formula \eqref{myrho}, we have
\[
\th\frac{1}{2}\rho_{ \mathbb{B}^2}(-s^{m/2},s^{m/2})
=\frac{2s^{m/2}}{s^m+1}.
\]
and hence
\begin{equation*} 
\mM(\tilde\Delta) = \frac{\pi}{\mu\left(2s^{m/2}/(s^m+1)\right)}.
\end{equation*}
By \cite[(7.20)]{HKV}
\[
\mu(s^m)= 2\mu\left(\frac{2 s^{m/2}}{1+s^m}\right)
\]
which together with the previous equality yields
\begin{equation}\label{eq:hT-4}
\mM(\tilde\Delta) = \frac{2\pi}{\mu\left(s^m\right)}.
\end{equation}
The proof then follows from~\eqref{eq:hT-1} and~\eqref{eq:hT-4}.
\end{proof}

\begin{rem}\label{limitCap}
The well-known inequality \cite[(7.21)]{HKV}
\[
\log(1/s)<\mu(s)<\log(4/s)
\]
implies that $\log (1/s)/\mu(s) \to 1$ when $s\to 0+$ and hence we see that in Lemma \ref{Lem:LB-Tri-h} 
\[
\lim_{m\to\infty} \frac{2 m \pi}{\mu(s^m)} = \frac{2\pi}{\log(1/s)},
\]
which is in accordance with the convergence \eqref{cap_convergence}
and the formula for the capacity of the annulus in Lemma \ref{cgqm_5.14}(1).
\end{rem}

\begin{lemma}
If $c(m,s)$ is the capacity of the condenser of Lemma \ref{Lem:LB-Tri-h} for $m\geq3$ and $0<s<1$, and $d(m,s)$ is the capacity of a segment with hyperbolic perimeter equal to that of the compact set $E$ in the condenser of Lemma \ref{Lem:LB-Tri-h}, then there are numbers $m$ and $s$ such that
\begin{align*}
d(m,s)\geq c(m,s).
\end{align*}
\end{lemma}
\begin{proof}
The hyperbolic perimeter of a segment $[0,r]$ is equal to twice of its hyperbolic diameter, so the value of the perimeter is by Lemma  \ref{perimBe}
\begin{align*}
4\,\text{arth}r=2\,\text{log}\frac{1+r}{1-r},\quad 0<r<1    
\end{align*}
and its capacity is by \eqref{capGro}
\begin{equation}\label{eq:seg-cap-s}
\capa(\B^2,[0,r])=2\pi\slash\mu(r).    
\end{equation}
From Lemma \ref{Lem:LB-Tri-h}, it follows that 
\begin{align*}
c(m,s)=\frac{2\pi m}{\mu(s^{m})}    
\end{align*}
and the perimeter of set $E$ in Lemma \ref{Lem:LB-Tri-h} is
\begin{align*}
2m \log \frac{1+s}{1-s}.
\end{align*}
For $t>0$, choose now $r$ and $s$ such that
\begin{align*}
t=4\,\text{arth}r=2m \log \frac{1+s}{1-s}.    
\end{align*}
Then for these values of $r$ and $s$, the capacities are
\begin{align*}
\capa(\B^2,[0,r])&=\frac{2\pi}{\mu(\th(t\slash4))},\\
\capa(\B^2,E)&=\frac{2m \pi}{\mu((\th(t\slash(4m)))^{m})}.
\end{align*}
Now, we claim that for some values of the parameters $s$ and $m$, equivalently $t$ and $m$, such that
\begin{align*}
\frac{2m \pi}{\mu((\th(t\slash(4m)))^{m})}<\frac{2\pi}{\mu(\th(t\slash4))},   
\end{align*}
and numerical computation shows that we can choose, for instance, $m=5$ and $s=0.5$.
\end{proof}

\begin{nonsec}{\bf The hypergeometric function and the SC transformation.}
Given complex numbers $a,b,c$ such that $c\neq0,-1,-2,\ldots$, the \emph{Gaussian hypergeometric function} is the analytic continuation to the slit plane $\C\backslash[1,\infty)$ defined by the series
\begin{equation}\label{eq:F}
F(a,b;c;z)=\sum^\infty_{n=0}\frac{(a,n)(b,n)}{(c,n)}\frac{z^n}{n!},\quad|z|<1.    
\end{equation}
Here, $(a,0)=1$ for $a\neq0$ and $(a,n)$ for $n=1,2,\ldots$ is the \emph{shifted factorial function}
\begin{align*}
(a,n)=a(a+1)(a+2)\cdots(a+n-1).    
\end{align*}
The Euler integral representation \cite{af,ky} 
\begin{equation} \label{int2F1}
 F (a, b; c; z) =\frac{\Gamma(b}{\Gamma(b)\Gamma(c -b)} \int_0^1 t^{b-1}(1 -t)^{c-b-1}
 (1 - tz)^{-a} dt \,
\end{equation}
links the  hypergeometric function with the conformal Schwarz-Christoffel transformation.
As shown in \cite{hvv}, this transformation delivers a conformal map of the upper half
plane onto a polygonal quadrilateral.
\end{nonsec}

\begin{theorem} \cite[Corollary 2.5]{hvv}\label{thm_quadriQ} 
Choose $a,b,c$ such that $0<a,b<1$ and $\max\{a+b,1\}\leq c\leq 1+\min\{a,b\}$. Let $Q$ be a polygonal quadrilateral in the upper half-plane with interior angles $b\pi$, $(c-b)\pi$, $(1-a)\pi$ and $(1+a-c)\pi$ at the vertices $0,1,A,B$, respectively. Then the conformal modulus of $Q$ is given by
\begin{align*}
\M(Q)=\frac{\K(r')}{\K(r)},    
\end{align*}
where $0<r<1$ fulfills the equation
\begin{align}
A-1=\frac{L(r')^{2(c-a-b)}F(c-a,c-b;c+1-a-b;r'^2)}{F(a,b;c;r^2)}=G(r)
\end{align}
with
\begin{equation}\label{eq:L}
L=\frac{B(c-b,1-a)}{B(b,c-b)}e^{(b+1-c)\pi i} 
\end{equation}
where $B(z,w)= \Gamma(z) \Gamma(w)/\Gamma(z+w)$ is the beta function.
\end{theorem}

\begin{remark} (1)
The hypotheses in Theorem \ref{thm_quadriQ} imposed on the triple $a,b,c$ imply that the quadrilateral $Q$ is convex.

(2) The algorithm of Theorem \ref{thm_quadriQ} will be implemented and applied for numerical
computation in the following sections. 
\end{remark}


\begin{nonsec}{\bf Polygonal ring domains.}\label{polyringdomain} 
A domain $G\subset\R^n$ is called a \emph{ring domain} if it is homeomorphic to the spherical annulus $B^n(t)\backslash\overline{B}^n(s)$ for some numbers $0<s<t$. We consider here planar domains characterized as follows: There exist closed convex quadrilaterals $P_j$, $j=1,\ldots,m$, with $\text{int}P_j\cap\text{int}P_k=\varnothing$ for all $j\neq k$ such that $P=\text{int}\left(\cup^m_{j=1}P_j\right)$ is a ring domain and its both boundary components are polygonal with $m$ vertices. Assume, moreover, that the inner boundary component of $P$ can be written as
\begin{align}
\bigcup^m_{j=1}I_j    
\end{align}
and the exterior boundary component as
\begin{align}
\bigcup^m_{j=1}E_j,    
\end{align}
where $I_j$ and $E_j$ are opposite sides of $P_j$. The set $P$ is now called a \emph{polygonal ring domain} (see Figure~\ref{fig:plg}). The degenerate case when the inner polygon is a segment, will be also studied in Section~\ref{sec:num}, see Figure~\ref{fig:rec-slit}.

\begin{figure}[H] %
\centerline{
\scalebox{0.6}{\includegraphics[trim=4cm 8cm 4cm 8cm,clip]{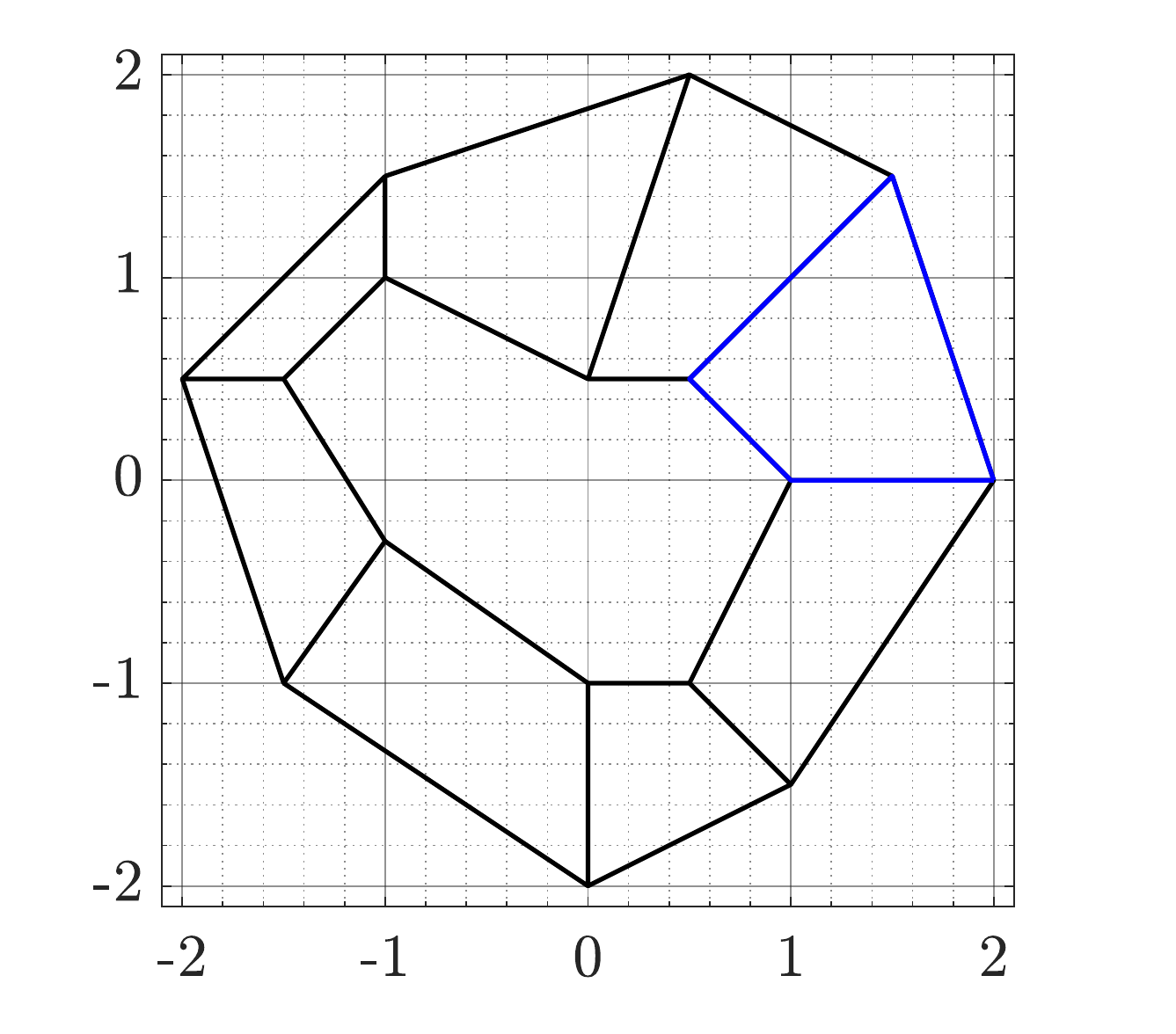}}
\hfill
\scalebox{0.6}{\includegraphics[trim=4cm 8cm 4cm 8cm,clip]{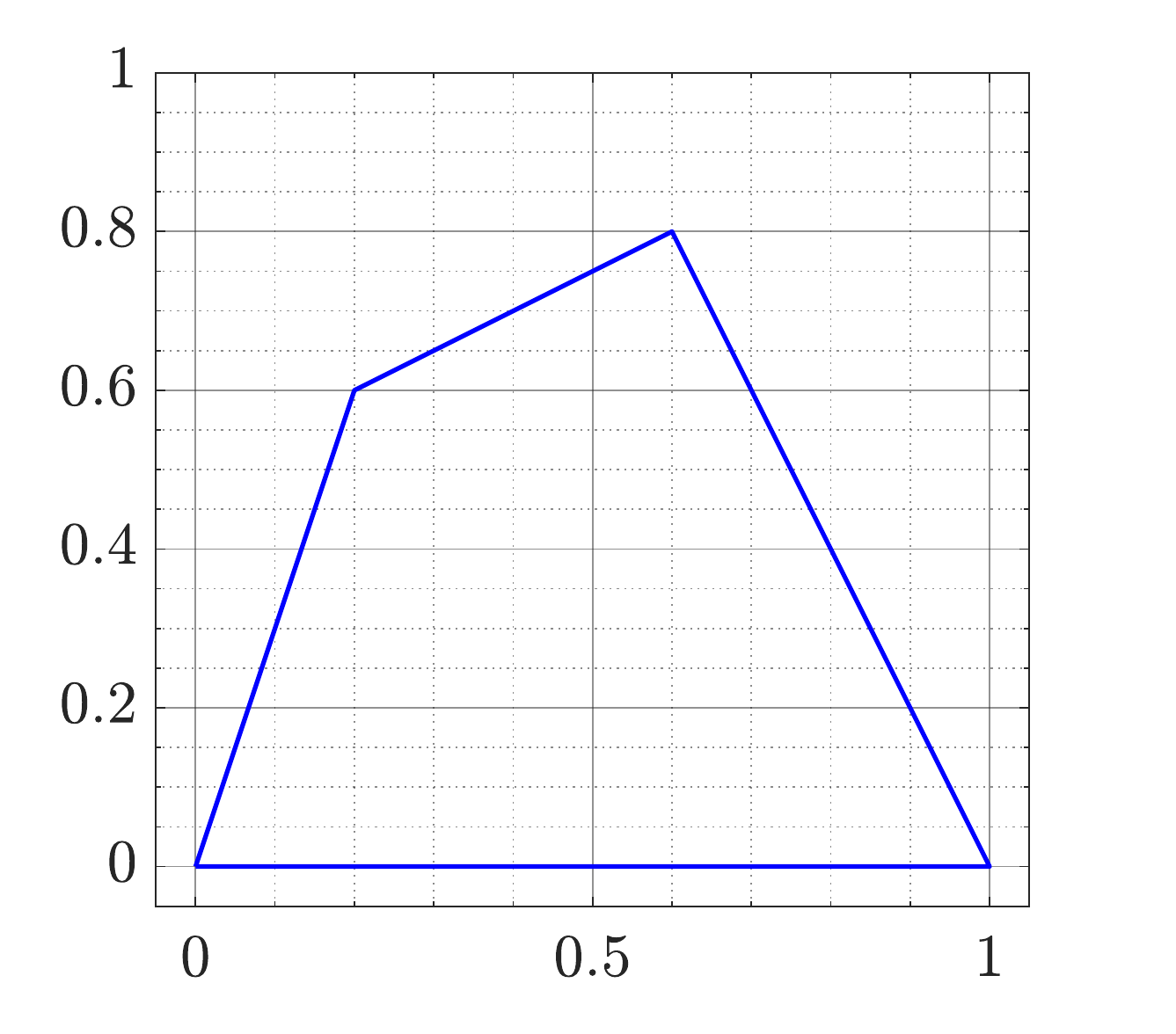}}
}
\caption{A polygonal ring domain with $m=8$ (left) and a trapezium quadrilateral (right).}
\label{fig:plg}
\end{figure}

\begin{lemma}\label{lem_Pjcapbound}
Let $P$ be a polygonal ring domain as above and denote $\Gamma_j=\Delta(I_j,E_j;P_j)$. Then
\begin{align*}
\capa(P,\cup^m_{j=1}I_j)\geq\sum^m_{j=1}\M(\Gamma_j).   
\end{align*}
\end{lemma}
\begin{proof}
The families of $\Gamma_j$ are separate subfamilies of 
\begin{align*}
\Delta\left(\bigcup^m_{j=1}I_j,\bigcup^m_{j=1}E_j;P\right)    
\end{align*}
and hence the proof follows from \cite[Thm 6.7]{v1}.
\end{proof}

The modulus $\M(\Gamma_j)$ can be computed numerically by mapping each of the convex quadrilaterals $P_j$ (see Figure~\ref{fig:plg} (left)) onto a quadrilateral which has two vertices on $0$ and $1$ as in Figure~\ref{fig:plg} (right) and then using the method presented in Theorem~\ref{thm_quadriQ}. Numerical results will be given in Subsection~\ref{sec:trapezium}.

\begin{remark}
Equality holds in the inequality of Lemma \ref{lem_Pjcapbound} in some symmetric situations. For instance, if $a_j=e^{2\pi ji\slash m}$, $E_j=[a_{j-1},a_j]$ and $I_j=[\lambda a_{j-1},\lambda a_j]$ with $j=1,\ldots,m$ and $0<\lambda<1$, the equality holds (see Figure~\ref{fig:reg-plg} below). If $m=4$ in this special case, an explicit analytic formula is known for the capacity, see \cite[7.45, pp. 404-406]{HKV}, \cite{nv}. 
\end{remark}

\end{nonsec}

\begin{nonsec}{\bf Quadrilaterals with three collinear vertices.}
We consider here quadrilaterals in the upper half-plane with the vertices $0,1,A,B$, where $A$ is a point on the segment $[1,B]$, i.e, $Q$ is a triangle with vertices at $0$, $1$, and $B$ (see Figure~\ref{fig:tri} (right)). The exact value of the conformal modulus of $Q$ can be obtained with the help of conformal mapping as in the following theorem. 

\begin{figure}[H] %
\centerline{
\scalebox{0.55}{\includegraphics[trim=0cm 0cm 0cm 0cm,clip]{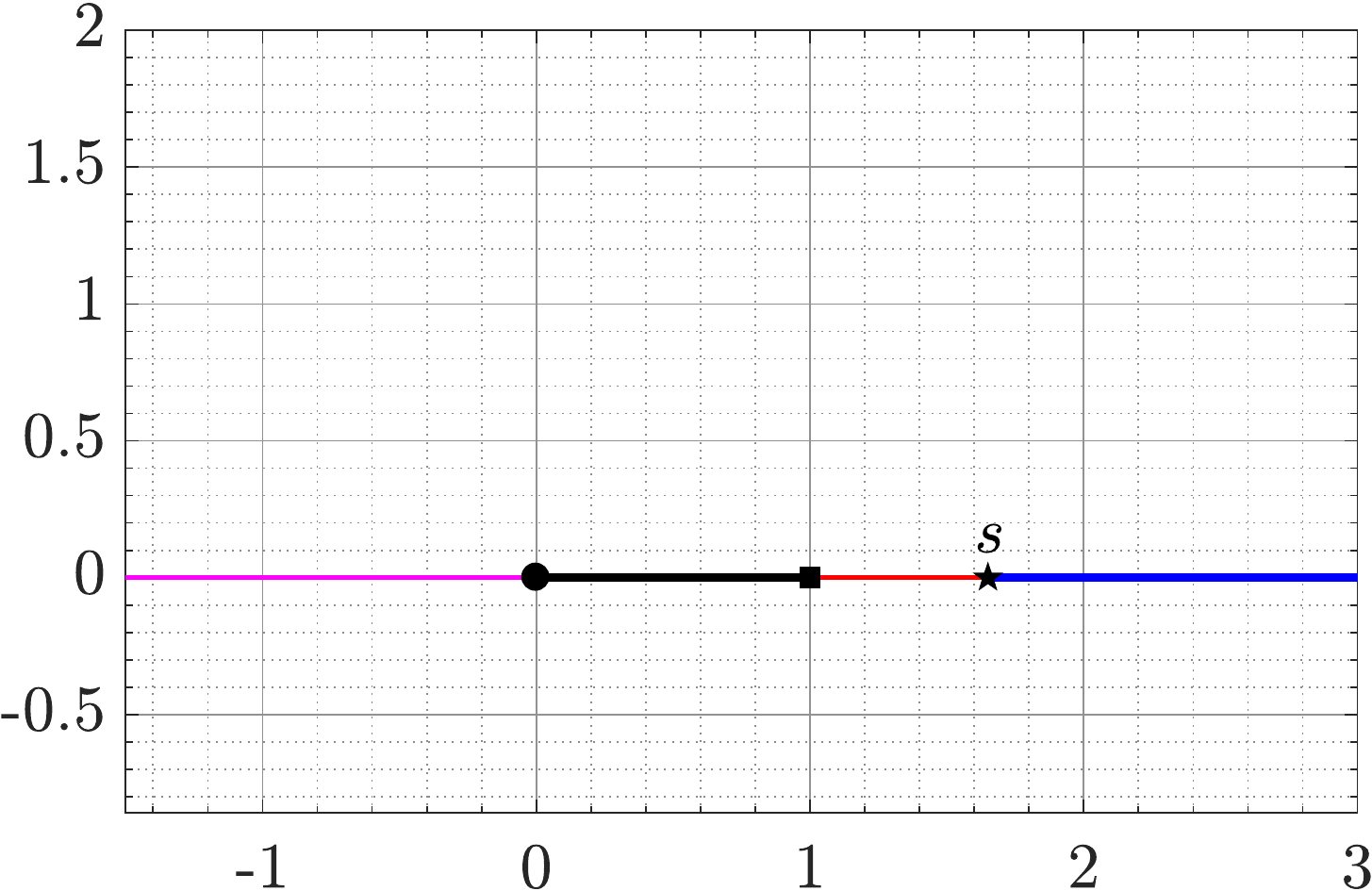}}
\hfill
\scalebox{0.55}{\includegraphics[trim=0cm 0cm 0cm 0cm,clip]{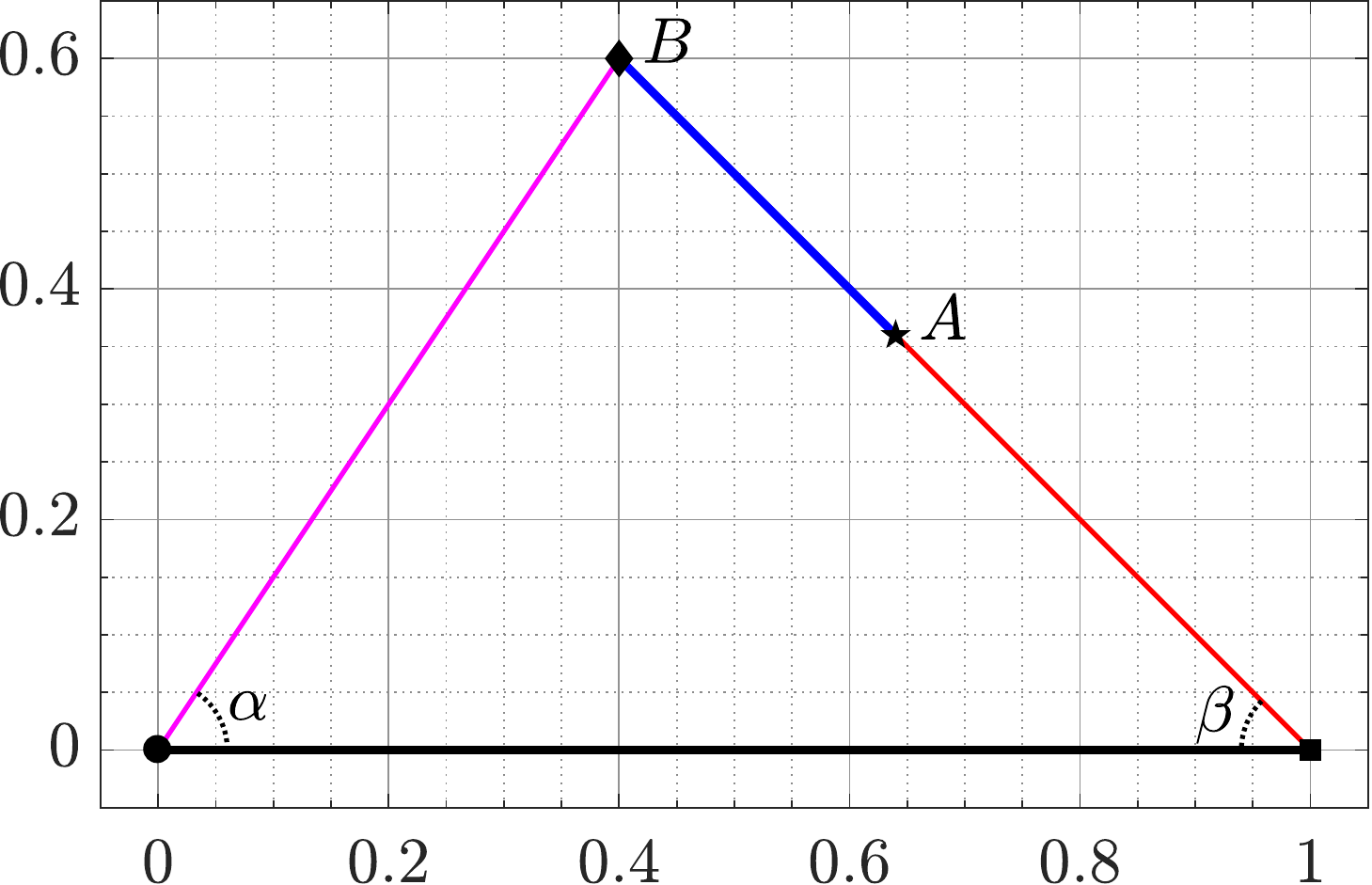}}
}
\caption{The upper half-plane and the quadrilateral with three collinear vertices.}
\label{fig:tri}
\end{figure}

\begin{theorem}\label{thm_quadriT} 
Let $Q$ be a quadrilateral in the upper half-plane with the vertices $0,1,A,B$, where $B$ is a point on the upper half-plane and $A$ is on the segment $[1,B]$, 
and let the interior angles at the vertices $0$ and $1$ be $\alpha$ and $\beta$, respectively.
Then the conformal modulus of $Q$ is given by
\begin{equation}\label{eq:half-tri-md1}
\M(Q)=\frac{\pi}{2\,\mu\left(\sqrt{1-1/s}\right)},
\end{equation}
where $1<s<\infty$ satisfies the  nonlinear real equation 
\begin{equation}\label{eq:map-h}
h(s) := e^{\i\beta}(f(s)-A)=0
\end{equation}
and
\begin{equation}\label{eq:map-f-T}
f(z) = \frac{\pi}{\alpha}\frac{\Gamma\left(\alpha/\pi+\beta/\pi\right)}{\Gamma(\alpha/\pi)\Gamma(\beta/\pi)}
\,z^{\alpha/\pi}\,F\left(\frac{\alpha}{\pi},1-\frac{\beta}{\pi};\frac{\alpha}{\pi}+1;z\right).
\end{equation}
\end{theorem}
\begin{proof}
By~\cite[p.~458]{shaw}, the function
\begin{equation}\label{eq:half-tri}
w = f(z) = \frac{\Gamma\left(\alpha/\pi+\beta/\pi\right)}{\Gamma(\alpha/\pi)\Gamma(\beta/\pi)}
B_z(\alpha/\pi,\beta/\pi)
\end{equation}
conformally maps the upper half-plane $\Im z>0$ onto the interior of the triangle $Q$ such that $f(0)=0$, $f(1)=1$, and $f(\infty)=B$ (see Figure~\ref{fig:tri}). Here $B_z$ is the incomplete beta function which can be written in terms of the Gaussian hypergeometric function $F$ as~\cite[6.6.8]{as}
\[
B_z(\alpha/\pi,1/2) = \frac{\pi}{\alpha}\,z^{\alpha/\pi}\,F\left(\frac{\alpha}{\pi},1-\frac{\beta}{\pi};\frac{\alpha}{\pi}+1;z\right).
\]
The mapping function $f$ is then given by~\eqref{eq:map-f-T}. By the conformal  invariance of the  conformal modulus, it follows from \cite[Lemma 7.12, 7.33(1)]{HKV} that
\begin{equation}\label{eq:half-tri-md2}
\M(Q)=\frac{1}{2}\tau_2\left(\frac{1}{s-1}\right),    
\end{equation}
where $s$ satisfies the equation $f(s)=A$ or, equivalently, $s$ satisfies the equation~\eqref{eq:map-h}. 
The conformal mapping $w=f(z)$ maps the infinite segment $(1,\infty)$ on the real line onto the the finite segment $[1,B]$ on the boundary of the quadrilateral $Q$. Thus, both $A$ and $f(s)$ for $1<s<\infty$ are on the segment $[1,B]$, and hence
\[
h(s) = e^{\i\beta}(f(s)-A) = e^{\i\beta}(f(s)-1)-e^{\i\beta}(A-1)
\]
is a real-valued function. 
Then by~\cite[(5.18)]{avv}, we obtain~\eqref{eq:half-tri-md1} from~\eqref{eq:half-tri-md2}.
\end{proof}

\end{nonsec}

\section{Numerical algorithms}\label{sec:num}

In this section, we describe numerical methods for computation of the capacity of condensers and the modulus of quadrilaterals.

\begin{nonsec}{\bf Computation of hyperbolic perimeter in the unit disk.}
If $E\subset\B^2$ is a continuum with piecewise smooth boundary, then the hyperbolic perimeter of $E$ is~\cite{be} 
\[
\mbox{h-perim}_{\B^2}(E)=\int_{\partial E}\frac{2|dz|}{1-|z|^2}.
\]
If the boundary $\partial E$ is parametrized by $\xi(t)$, $0\le t\le2\pi$, then 
\[
\mbox{h-perim}_{\B^2}(E)=\int_{0}^{2\pi}\frac{2|\xi'(t)|}{1-|\xi(t)|^2}dt.
\]
The integrand is $2\pi$-periodic and hence it can be accurately approximated by the trapezoidal rule~\cite{dr} to obtain
\[
\mbox{h-perim}_{\B^2}(E)\approx \frac{4\pi}{n}\sum_{k=1}^{n}\frac{|\xi'(s_k)|}{1-|\xi(s_k)|^2},
\]
where
\begin{equation}\label{eq:sk}
s_k = (k-1) \frac{2 \pi}{n}, \quad k = 1, \ldots, n,
\end{equation}
and $n$ is an even integer.

\end{nonsec}

\begin{nonsec}{\bf Computation of hyperbolic perimeter in simply connected domains.}\label{sec:hyp-diam}
As above, let $E\subset G$ be a continuum in a simply connected domain $G$ and $f\,:\,G\to \B^2=f(G)$ be a conformal map. Then $f$ maps the connected set $E$ onto a connected set $\hat E\subset\B^2$. 
If $\mbox{h-perim}_{G}(E)$ is the hyperbolic perimeter of $E$ with respect to the hyperbolic metric $\rho_G$ in $G$, then $\mbox{h-perim}_G(E)=\mbox{h-perim}_{\B^2}(\hat E)$. 
Furthermore, if the boundary $\partial E$ is parametrized by $\xi(t)$, $0\le t\le2\pi$, then $\partial\hat E$ is parametrized by $\zeta(t)=f(\xi(t))$, $0\le t\le2\pi$. 
The parametrization $\zeta(t)$ is computed by approximating numerically the conformal mapping $f$, which is done here by the numerical method presented in~\cite{Nas-cmft15,nvs}. The derivative $\zeta'(t)$ is computed by approximating the real and imaginary parts of $\zeta(t)$ by trigonometric interpolating polynomials and then differentiating the interpolating polynomials. These polynomials can be computed using FFT~\cite{Weg05}. In terms of  $\zeta(t)$ and $\zeta'(t)$, we can compute $\mbox{h-perim}_{G}(E)$ by
\[
\mbox{h-perim}_{G}(E)=\mbox{h-perim}_{\B^2}(\hat E)
=\int_{0}^{2\pi}\frac{2|\zeta'(t)|}{1-|\zeta(t)|^2}dt,
\]
which can be approximated by the trapezoidal rule to obtain
\[
\mbox{h-perim}_{G}(E)\approx \frac{4\pi}{n}\sum_{k=1}^{n}\frac{|\zeta'(s_k)|}{1-|\zeta(s_k)|^2},
\]
where $s_1, \ldots, s_n$ are as in~\eqref{eq:sk}.
\end{nonsec}

\begin{nonsec}{\bf Algorithm for the capacity of a polygonal ring domain.}
Consider a bounded simply connected domain $G$ in the complex plane and a compact set $E\subset G$ such that $D= G\setminus E$ is a doubly connected domain. 
In this paper, the capacity of the condenser $(G,E)$ will be computed by the MATLAB function \verb|annq| from~\cite{nv}. In this function \verb|annq|, the capacity is computed by a fast method based on an implementation of the Fast Multipole Method toolbox~\cite{Gre-Gim12} in solving the boundary integral equation with the generalized Neumann kernel~\cite{Nas-ETNA}.

We assume that the boundary components of $D= G\setminus E$ are piecewise smooth Jordan curves.
Let $\Gamma_1$ be the external boundary component and $\Gamma_2$ be the inner boundary component such that $\Gamma_1$ is oriented counterclockwise and $\Gamma_2$ is oriented clockwise.
We parametrize each boundary component $\Gamma_j$ by a $2\pi$-periodic function $\eta_j(\delta_j(t))$, $0\le t\le 2\pi$, where $\delta_j\,:\,[0,2\pi]\to[0,2\pi]$ is a bijective strictly monotonically increasing function and $\eta_j$ is a $2\pi$-periodic parametrization of $\Gamma_j$, which is assumed to be smooth except at the corner points. 
The function $\delta_j$ is introduced to remove the singularity in the solution of the integral equation at the corner points~\cite{kre90}.
When $\Gamma_j$ is smooth, we assume $\delta_j(t)=t$, $0\le t\le 2\pi$. 
If $\Gamma_j$ has corners, we choose the function $\delta_j$ as in~\cite[p.~697]{LSN17}. 
Then, we define the vectors \texttt{et} and \texttt{etp} in MATLAB by
\begin{eqnarray*}
\texttt{et} &=& [\eta_1(\delta_1(\bs))\,;\,\eta_2(\delta_2(\bs))]\in\C^{2n}, \\
\texttt{etp} &=& [\eta_1'(\delta_1(\bs))\delta_1'(\bs)\,;\,\eta_2'(\delta_2(\bs))\delta_2'(\bs)]\in\C^{2n},  
\end{eqnarray*}
where $\bs=[s_1,\ldots, s_n]^T\in\R^n$ and $s_1, \ldots, s_n$ are given by~\eqref{eq:sk}. The MATLAB function \verb|annq| is then used to approximate $\capa(G,E)$ as follows,
\begin{verbatim}
  [~,cap] = annq(et,etp,n,alpha,z2,'b'),
\end{verbatim}
where $\alpha$ is an auxiliary point in the domain $D$ and $z_2$ is an auxiliary point in the interior of $E$. The values of the parameters in the function \verb|annq| are chosen as in~\cite{nv}.

The readers are referred to~\cite{LSN17,Nas-ETNA,nv} for more details.

\end{nonsec}

\begin{nonsec}{\bf Algorithm for the modulus of quadrilateral.}
In this section, we present a MATLAB implementation of the methods presented in Theorems~\ref{thm_quadriQ} and~\ref{thm_quadriT}. 

Let $\QM(A,B)$ be the modulus of the quadrilateral with the vertices $0,1,A,B$ as described in Theorem~\ref{thm_quadriQ}. Then by symmetry with respect to the line ${\rm Re} z=1/2$,  we have
\begin{equation} \label{symm}
\QM(A,B) =\QM(-\overline{B}+1, -\overline{A}+1)\,,
\end{equation}
and it follows from  \eqref{recipIdty} that
\begin{equation} \label{RecipFormula}
\QM(A,B) = \frac{2}{\pi}\mu(r) = 1/ \QM((B-1)/(A-1),-1/(A-1))\,.
\end{equation}
Here $\mu(r)$ is defined by~\eqref{capGro}, $r'=\sqrt{1-r^2}$, and 
$r$ is the solution of the nonlinear equation $h(r)=0$ where
\[
h(r)=\frac{(r')^{2(c-a-b)}F(c-a,c-b;c+1-a-b;r'^2)}{F(a,b;c;r^2)}-\frac{A-1}{L}
\]
and $L$ is given by~\eqref{eq:L}.
The equation $h(r)=0$ is solved for $r\in(0,1)$ using the MATLAB function \verb|fzero| if $f(10^{-6})\times f(1-10^{-13})<0$. If $f(10^{-6})\times f(1-10^{-13})>0$, an approximate value to the solution $r$ of the equation $h(r)=0$ is computed by minimizing the function $h^2(r)$ using the MATLAB function \verb|fminbnd|. The values of $\mu(r)$ are then computed as described in~\cite{nv}. This method for computing $\QM(A,B)$ is implemented in MATLAB as in the following code, which is based on the Mathematica code presented in~\cite{hvv}. 

\begin{lstlisting}
function md = QM(A,B)
beta = @(x,y)(gamma(x)*gamma(y)/gamma(x+y));
a =  1-(angle(A-1)-angle(A-B))/pi; b =  angle(B)/pi;
c = (pi-angle(A-1)+angle(B))/pi;
L = (beta(c-b,1-a)/beta(b,c-b))*exp(i*(b+1-c)*pi);
f = @(x)(h(x)); ff= @(x)(h(x)^2);
if f(1e-6)*f(1-1e-13)<0
    r = fzero(f,[1e-6,1-1e-13]);
else
    r = fminbnd(ff,0,1,optimset('TolX',1e-12));
end
md = (2/pi)*mu(r);
% 
function y = h(x)
   nn = ((1-x^2)^(c-a-b))*hypergeom([c-a,c-b],c+1-a-b,1-x^2);
   dd =  hypergeom([a,b],c,x^2);
    y =  nn/dd-real((A-1)/L);
end
end
\end{lstlisting}

The above MATLAB function \verb|QM| is tested with the following example. Let $B=-5+2\i$ and let $z=x+\i y$ be any point in the rectangle $[-1,3]\times[1,3]$. We define the function
\[
u(x,y)=\left|\QM(z,B)\cdot \QM((B-1)/(z-1),-1/(z-1))-1\right|.
\]
The function $u$ should be identically zero by \eqref{RecipFormula}. The surface plot of this function $u$ is presented in Figure~\ref{fig:QM-surf}. The figure shows that the maximum value of $u(x,y)$ over the rectangle is of the order $10^{-14}$.

\begin{figure}[H] %
\centerline{
\scalebox{0.6}{\includegraphics[trim=0cm 0cm 0cm 0cm,clip]{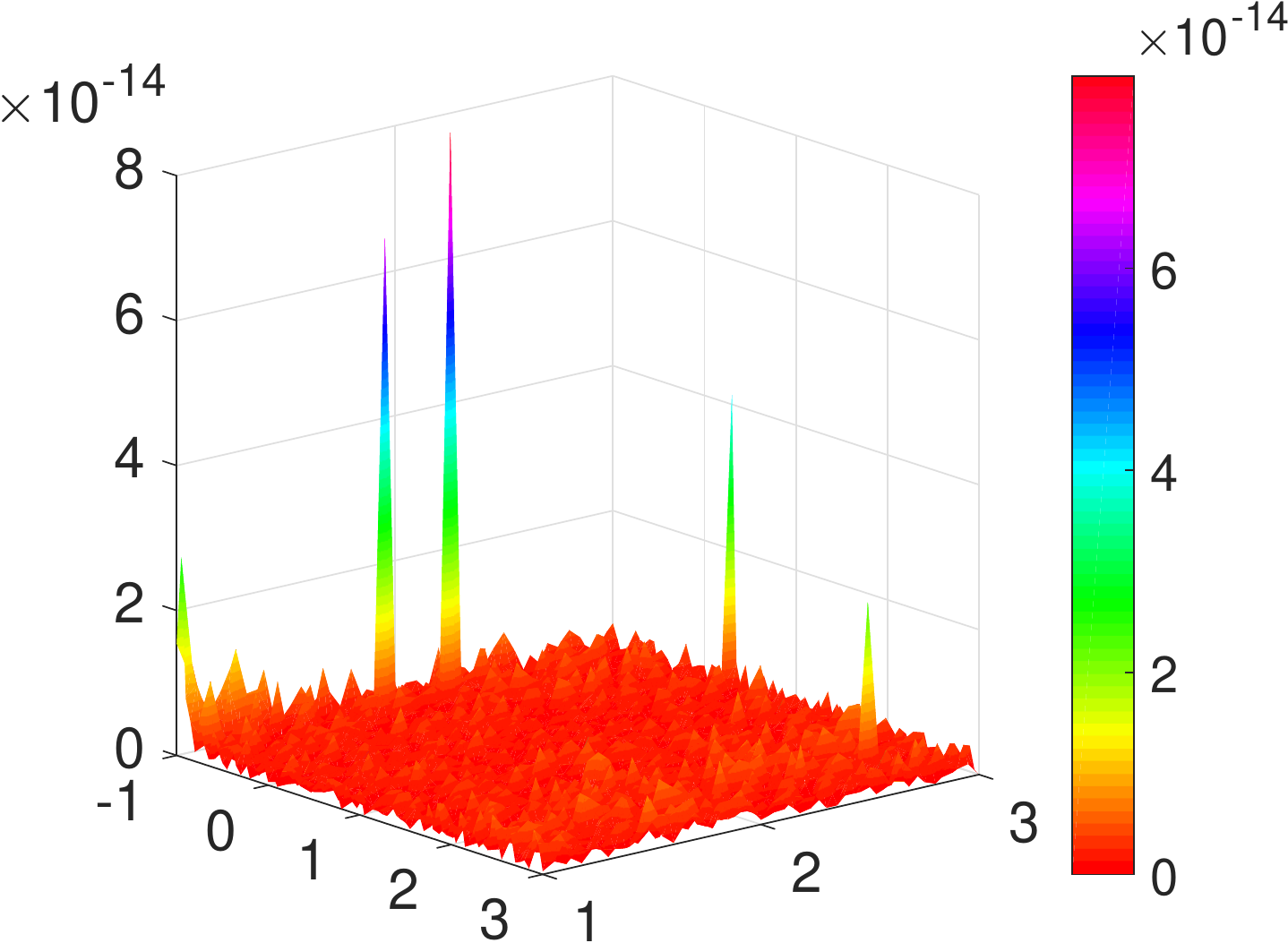}}
}
\caption{A surface plot of the function $u(x,y)$ over $[-1,3]\times[1,3]$.}
\label{fig:QM-surf}
\end{figure}

The MATLAB function \verb|QM| is also tested for several values of $A$ and $B$ as in Table~\ref{tab:quad-AB} where, in view of~\eqref{RecipFormula}, the error is computed by
\[
\left|\QM(A,B)\cdot \QM((B-1)/(A-1),-1/(A-1))-1\right|.
\]

\begin{table}[hbt]
\caption{The values of $\QM(A,B)$ for several values of $A$ and $B$.}
\label{tab:quad-AB}%
\begin{tabular}{c|c|c|c}\hline
$A$      & $B$      & $\QM(A,B)$         & Error    \\ \hline
$ 7+5\i$ & $-1+2\i$ & $1.17336589158553$ & $1.1102\times 10^{-16}$ \\
$ 8+3\i$ & $-1+1\i$ & $0.71853428024898$ & $4.6629\times 10^{-15}$ \\
$ 5+5\i$ & $-3+1\i$ & $1.00171178298845$ & $2.4425\times 10^{-15}$ \\
$ 7+4\i$ & $-3+3\i$ & $1.17821610141750$ & $6.6613\times 10^{-16}$ \\
$ 5+5\i$ & $-1+2\i$ & $1.27382477147819$ & $8.8818\times 10^{-16}$ \\
$ 7+5\i$ & $ 0+1\i$ & $0.92223220304256$ & $8.8818\times 10^{-16}$ \\
$ 7+3\i$ & $ 1+2\i$ & $1.68574560877551$ & $2.5535\times 10^{-15}$ \\
$ 4+5\i$ & $-2+1\i$ & $1.02479880902234$ & $1.9984\times 10^{-15}$ \\
\hline %
\end{tabular}
\end{table}

For the method presented in Theorem~\ref{thm_quadriT}, let $\QMt(A,B)$ be the modulus of the quadrilateral with the vertices $0,1,A,B$, where $B$ is a point on the upper half-plane and $A$ is on the segment $[1,B]$. Then, a MATLAB code for computing the values of $\QMt(A,B)$ can be written as follows, where the nonlinear real equation $h(s)=0$ is solved using the MATLAB function \verb|fzero|. 

\begin{lstlisting}
function md = QMt(A,B)
alp = angle(B);       alppi = alp/pi;
bet = pi-angle(-1+B); betpi = bet/pi;
map = @(z)((gamma(alppi+betpi)/(gamma(alppi)*gamma(betpi)*alppi)).*...
            z.^alppi.*hypergeom([alppi,1-betpi],alppi+1,z));
h   = @(s)(real(exp(i*bet)*(map(s+i*1e-100)-A)));
c  = 1; d  = 10;
while h(c)*h(d)>0 d=2*d; end
s  = fzero(h,[c,d])
md = 0.5*pi/mu(sqrt(1-1/s));
end
\end{lstlisting}

\end{nonsec}

\begin{nonsec}{\bf Modulus of isosceles trapezoid.}\label{sec:trapezoid}
Let the convex quadrilateral $P_j$ and let the segments $E_j=[a_{j-1},a_j]$ and $I_j=[b_{j-1},b_j]$, $j=1,2,\ldots,m$, be as described in Subsection~\ref{polyringdomain}. We consider first the case of a regular polygon, i.e., we assume that
\[
a_j=e^{2\pi ji\slash m}, \quad b_j=\lambda e^{2\pi ji\slash m}, \quad j=0,1,\ldots,m,
\]
and hence $a_0=a_m$ and $b_0=b_m$ (see Figure~\ref{fig:reg-plg} (left)). Let $\Gamma_j=\Delta(I_j,E_j;P_j)$. Then, for this case, we have
\[
\M(\Gamma_1)=\M(\Gamma_2)=\cdots=\M(\Gamma_m),
\]
and
\[
\capa\left(P,\bigcup^m_{j=1}I_j\right)=\sum^m_{j=1}\M(\Gamma_j)=m\,\M(\Gamma_1).
\]

\begin{figure}[H] %
\centerline{
\scalebox{0.6}{\includegraphics[trim=4cm 8cm 4cm 8cm,clip]{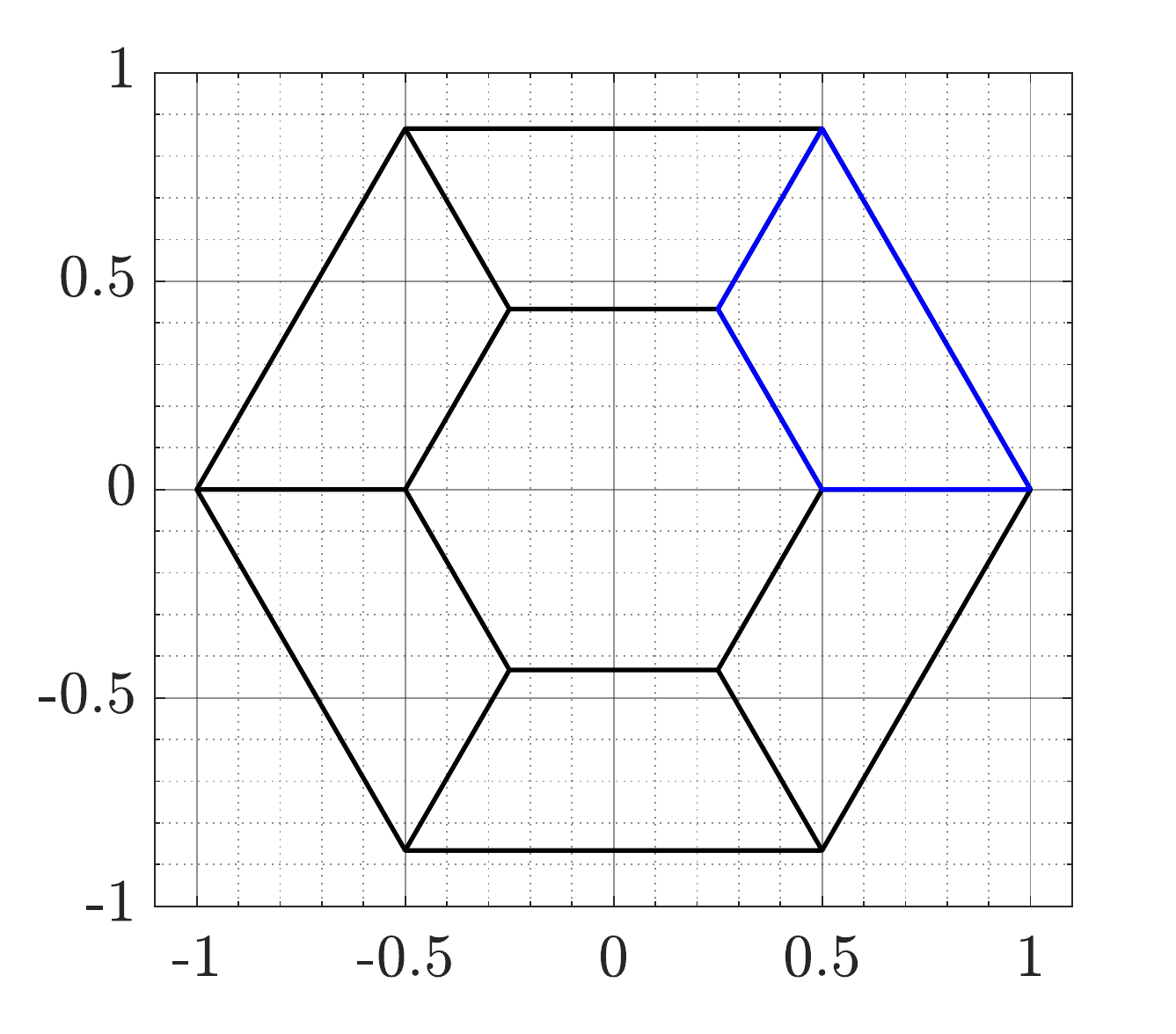}}
\hfill
\scalebox{0.6}{\includegraphics[trim=4cm 8cm 4cm 8cm,clip]{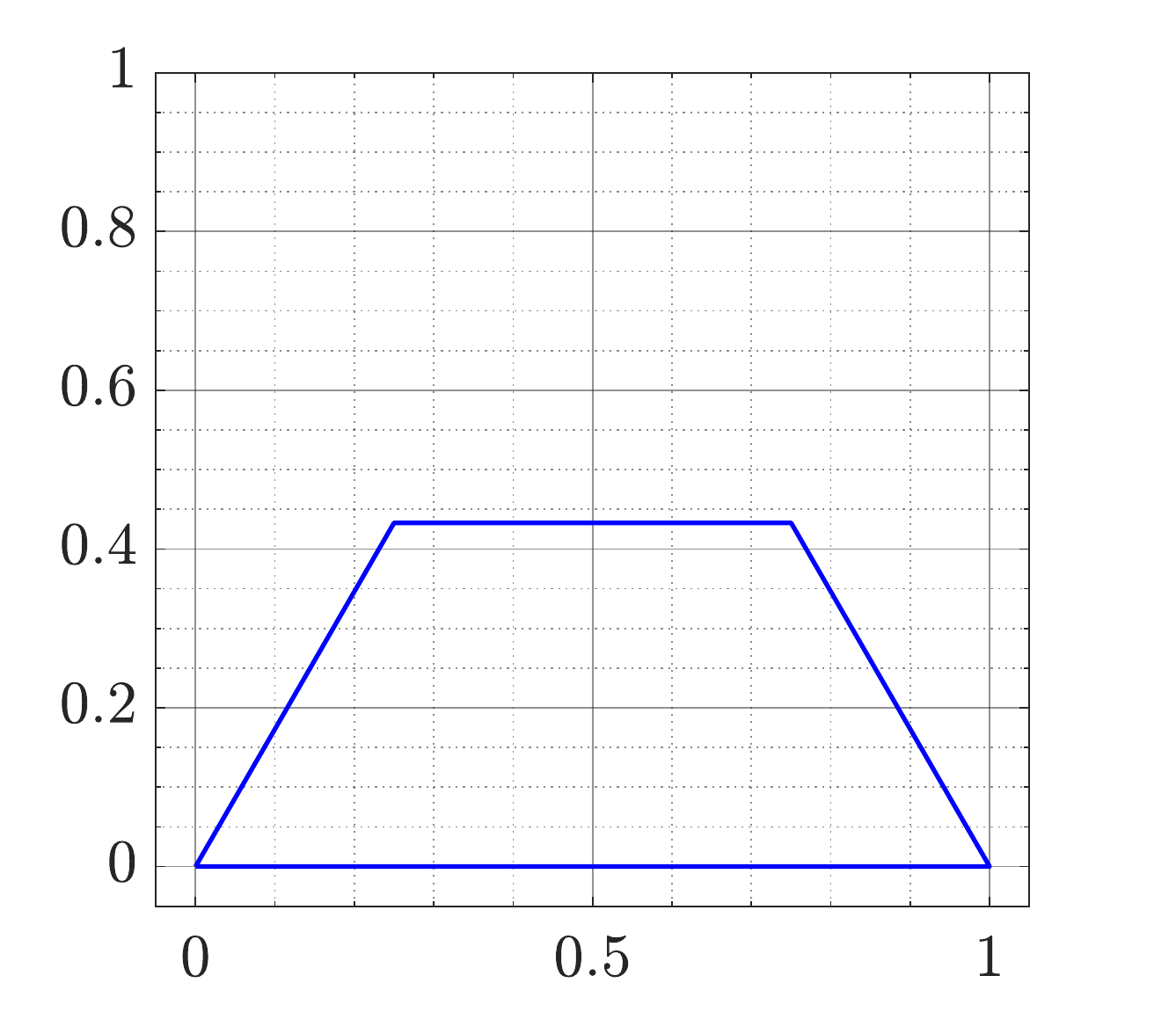}}
}
\caption{Isosceles trapezoid quadrilaterals for $m=6$ and $\lambda=0.5$.}
\label{fig:reg-plg}
\end{figure}

To compute $\M(\Gamma_1)$, we use the linear transformation
\[
z\mapsto \frac{z-1}{e^{2\pi\i\slash m}-1}
\]
to map the quadrilateral $P_1$ onto the quadrilateral $Q_1$ with the vertices $0$, $1$, $A$, and $B$, where (see Figure~\ref{fig:reg-plg} (right))
\[
A=\frac{\lambda e^{2\pi\i\slash m}-1}{e^{2\pi\i\slash m}-1}
=\frac{1+\lambda}{2}+\i\frac{1-\lambda}{2}\cot\frac{\pi}{m}
\quad {\rm and}\quad 
B=\frac{\lambda -1}{e^{2\pi\i\slash m}-1}
=\frac{1-\lambda}{2}+\i\frac{1-\lambda}{2}\cot\frac{\pi}{m}.
\]
Then $\M(\Gamma_1)=1/\M(Q_1)$, and hence 
\[
\capa\left(P,\bigcup^m_{j=1}I_j\right)=\frac{m}{\M(Q_1)}.
\]
The values of $\M(Q_1)$ are computed using the above MATLAB function \verb|QM| and these values are considered as exact values. The values of $\capa(P,\cup^m_{j=1}I_j)$ are computed numerically using the MATLAB function \verb|annq|. The absolute error between the values of $m/\M(Q_1)$ and the approximate values of $\capa(P,\cup^m_{j=1}I_j)$ for several values of $m$ and $\lambda$ are presented in Figure~\ref{fig:quad-reg-err}. Table~\ref{tab:quad-reg-err} presents the values of $\capa(P,\cup^m_{j=1}I_j)$ obtained with the function \verb|annq| for several values of $m$ and $\lambda$. For $m=\infty$, the outer boundary of $P$ reduces to the unit circle $|z|=1$ and the inner boundary reduces to the circle $|z|=\lambda$ and hence the capacity is $2\pi/\log(1/\lambda)$.

\begin{figure}[H] %
\centerline{
\scalebox{0.6}{\includegraphics[trim=3cm 8.5cm 3cm 8.5cm,clip]{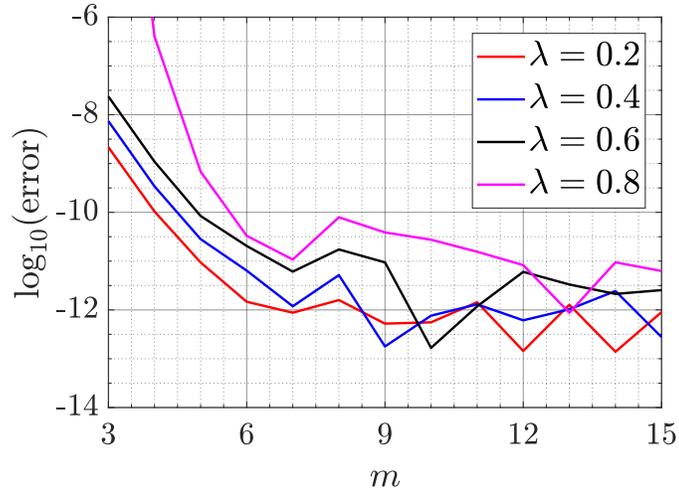}}
}
\caption{The absolute error between the values of $m/\M(Q_1)$ and $\capa(P,\cup^m_{j=1}I_j)$.}
\label{fig:quad-reg-err}
\end{figure}

\begin{table}[hbt]
\caption{The values of $\capa(P,\cup^m_{j=1}I_j)$ for several values of $m$ and $\lambda$.}
\label{tab:quad-reg-err}%
\begin{tabular}{l|l|l|l|l}\hline
$m\backslash \lambda$ & $0.2$ & $0.4$           & $0.6$           & $0.8$    \\ \hline
$3$   &$4.62006340262352$ &$8.97678254687922$ &$17.6373173090395$ &$43.6180795365658$ \\
$4$   &$4.13448702413021$ &$7.5615315394701$  &$14.2348796747326$ &$34.2349151937253$ \\
$5$   &$4.01100862917758$ &$7.18779448149662$ &$13.2642395249519$ &$31.4282920582438$ \\
$6$   &$3.96305925819319$ &$7.04023992761479$ &$12.855048267353$  &$30.1775839605043$ \\
$7$   &$3.94020862840696$ &$6.96948190383544$ &$12.6488960100174$ &$29.5090017429693$ \\
$8$   &$3.92785260938096$ &$6.9311587605797$  &$12.5332808502239$ &$29.1111436472386$ \\
$9$   &$3.92056195402897$ &$6.90855032892899$ &$12.4634416665474$ &$28.8568206239776$ \\
$10$  &$3.91597345356159$ &$6.89433186945894$ &$12.4188286559575$ &$28.6856579890056$ \\
$\infty$  
&$3.90396253166234$ 
&$6.85719618087606$ 
&$12.3000589924555$
&$28.1575930389859$\\
\hline %
\end{tabular}
\end{table}
           
\end{nonsec}

\begin{nonsec}{\bf Modulus of trapezium.}\label{sec:trapezium}
Let the convex quadrilateral $P_j$ and the segments $E_j$ and $I_j$, $j=1,2,\ldots,m$, be as described in Subsection~\ref{polyringdomain}. 
Here, for $j=1,2,\ldots,m$, we choose
\[
a_j=(3-0.5\lambda_j) e^{\i\theta_j}, \quad \theta_j=(j-1.2+0.4\tau_j)\frac{2\pi}{m},
\]
\[
b_j=(1+0.5\hat\lambda_j) e^{\i\hat\theta_j}, \quad \hat\theta_j=(j-1.2+0.4\hat\tau_j)\frac{2\pi}{m},
\]
where $\lambda_j$, $\hat\lambda_j$ $\tau_j$, and $\hat\tau_j$ are random real numbers in $(0,1)$.
Assume $E_j=[a_{j-1},a_j]$ and  $I_j=[b_{j-1},b_j]$, where $a_0=a_m$ and $b_0=b_m$. 
With the linear transformation
\[
z\mapsto (z-a_{j-1})/(a_j-a_{j-1}),
\]
the quadrilateral $P_j$ is mapped onto the quadrilateral $Q_j$ with the vertices $0$, $1$, $A_j$, and $B_j$ where $A_j=(b_j-a_{j-1})/(a_j-a_{j-1})$ and $B_j=(b_{j-1}-a_{j-1})/(a_j-a_{j-1})$, $j=1,2,\ldots,m$ (see Figure~\ref{fig:plg}). 
Let $\Gamma_j=\Delta(I_j,E_j;P_j)$ so that 
\[
\M(\Gamma_j)=1/\M(Q_j), \quad j=1,2,\ldots,m.
\]
Hence, by Lemma~\ref{lem_Pjcapbound}, we have
\begin{equation}\label{eq:cap-Qj}
\capa\left(P,\bigcup^m_{j=1}I_j\right)\ge\sum_{j=1}^m\frac{1}{\M(Q_j)}.
\end{equation}

For several values of $m$, the values of $\M(Q_j)$ are computed using the MATLAB function \verb|QM| and the values of $\capa(P,\cup^m_{j=1}I_j)$ are computed numerically using the MATLAB function \verb|annq|.
The obtained numerical results are presented in Figure~\ref{fig:quad-plg-err}.
These results validate the inequality~\eqref{eq:cap-Qj}. 

\begin{figure}[H] %
\centerline{
\scalebox{0.6}{\includegraphics[trim=3cm 8.5cm 3cm 8.5cm,clip]{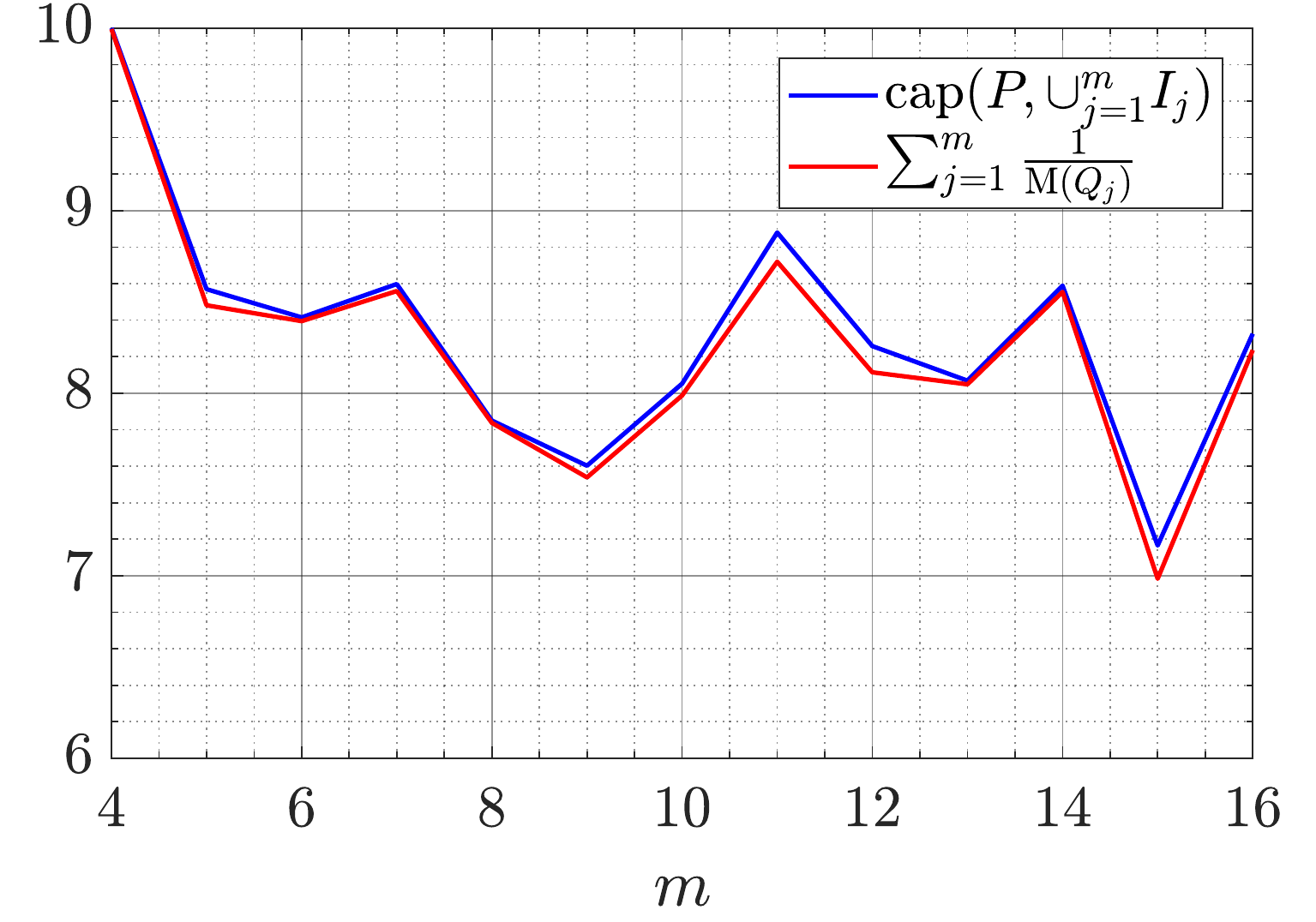}}
}
\caption{The values of $\sum_{j=1}^m(1/\M(Q_j))$ and $\capa(P,\cup^m_{j=1}I_j)$.}
\label{fig:quad-plg-err}
\end{figure}

\end{nonsec}

\begin{nonsec}{\bf Convex Euclidean polygons.}\label{sec:conv-plg}
Assume that $E\subset\B^2$ is a convex Euclidean polygonal closed region such that $\partial E$ is a polygon with $m$ vertices $v_{1}$, $v_{2}$, $\ldots$, $v_{m}$ where $m$ is an integer chosen randomly such that $3\le m\le 12$. We choose a real number $s$ randomly such that $0.05< s< 0.95$, then we assume the vertices are given by
\begin{equation}\label{eq:conv-thteta}
v_{j} = s\,e^{\i \theta_j}, \quad \theta_j = \left(j-1.25+0.5\tau_j\right)\frac{2\pi}{m}, \quad j=1,2,\ldots,m,
\end{equation}
where $\tau_j$ is a random number on $(0,1)$ (see Figure~\ref{fig:euc-plg-conv} for $m=5$ and $s=0.6973$). The values of $L=\mbox{h-perim}_{\B^2}(\partial E)$ are computed by the method presented in Subsection~\ref{sec:hyp-diam}.
Let $D=B^2(0,R)$ be the Euclidean disk with $\mbox{h-perim}_{\B^2}(\partial D)=L$. Then, it follows from~\eqref{eq:disk-cap-P} that
%
\begin{equation}\label{eq:cap-D}
\capa(\B^2,D) = \frac{2\pi}{\log\left(\sqrt{1+(2\pi/L)^2}+2\pi/L\right)}.
\end{equation}
Finally, let $I=[0,a]$, where $a$ is a positive real number such that $\mbox{h-perim}_{\B^2}(I)=L$, i.e., $I$ has the same hyperbolic perimeter as $\partial E$ as well as $\partial D$. Since the hyperbolic length of $I$ is $2\arth(a)$, the hyperbolic perimeter of $I$ is $4\arth(a)$ and the number $a$ must satisfy $4\arth(a)=L$. Thus, the number $a$ is given by $a = \tth(L/4)$.
Hence, by~\eqref{eq:seg-cap-s}, the exact value of $\capa(\B^2,I)$ is known and it is given by
\begin{equation}\label{eq:cap-I}
\capa(\B^2,I) = \frac{2\pi}{\mu(a)}  = \frac{2\pi}{\mu(\tth(L/4))}.
\end{equation}
The capacity $\capa(\B^2,E)$ is calculated using the MATLAB function \verb|annq|. 

\begin{figure}[H] %
\centerline{
\scalebox{0.6}{\includegraphics[trim=0cm 0cm 0cm 0cm,clip]{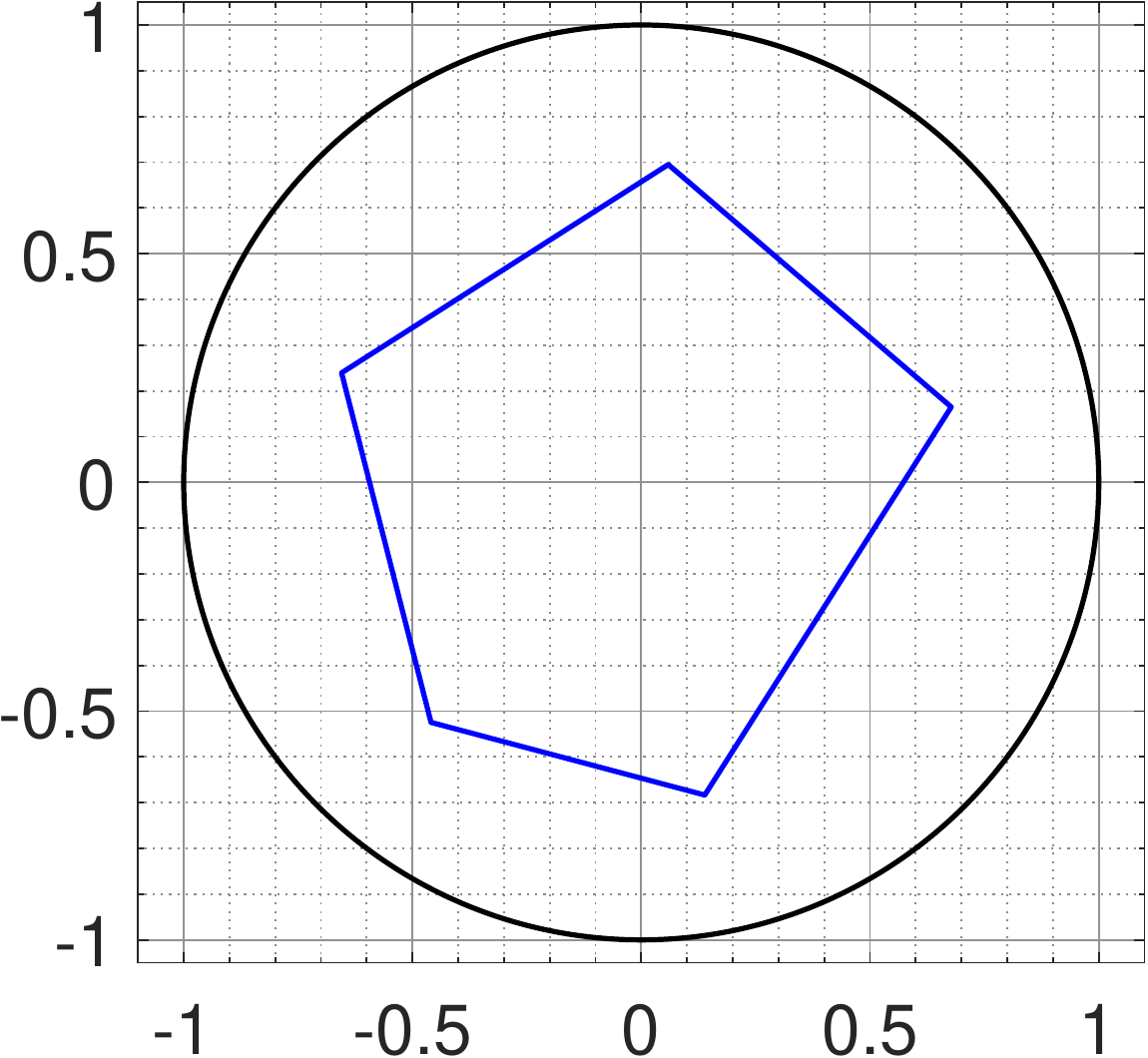}}
}
\caption{ A convex polygonal domain for $m=5$ and $s=0.6973$.}
\label{fig:euc-plg-conv}
\end{figure}

The values of the capacities $\capa(\B^2,E)$, $\capa(\B^2,D)$, and $\capa(\B^2,I)$ versus the hyperbolic perimeter $L$ where $L=\mbox{h-perim}_{\B^2}(\partial E)=\mbox{h-perim}_{\B^2}(\partial D)=\mbox{h-perim}_{\B^2}(I)$, are given in Figure~\ref{fig:euc-plg-conv-res1}. These values are computed for $200$ random values of $m$ and $s$. It is clear from Figure~\ref{fig:euc-plg-conv-res1} that 
\begin{equation}\label{eq:ineq-EDI}
\capa(\B^2,I)\le\capa(\B^2,E)\le\capa(\B^2,D),
\end{equation}
for the above described convex set $E$. These results agree with Theorems~\ref{KuhnauThm}
 and~\ref{fwg71}.

\begin{figure}[H] %
\centerline{
\scalebox{0.6}{\includegraphics[trim=0cm 0cm 0cm 0cm,clip]{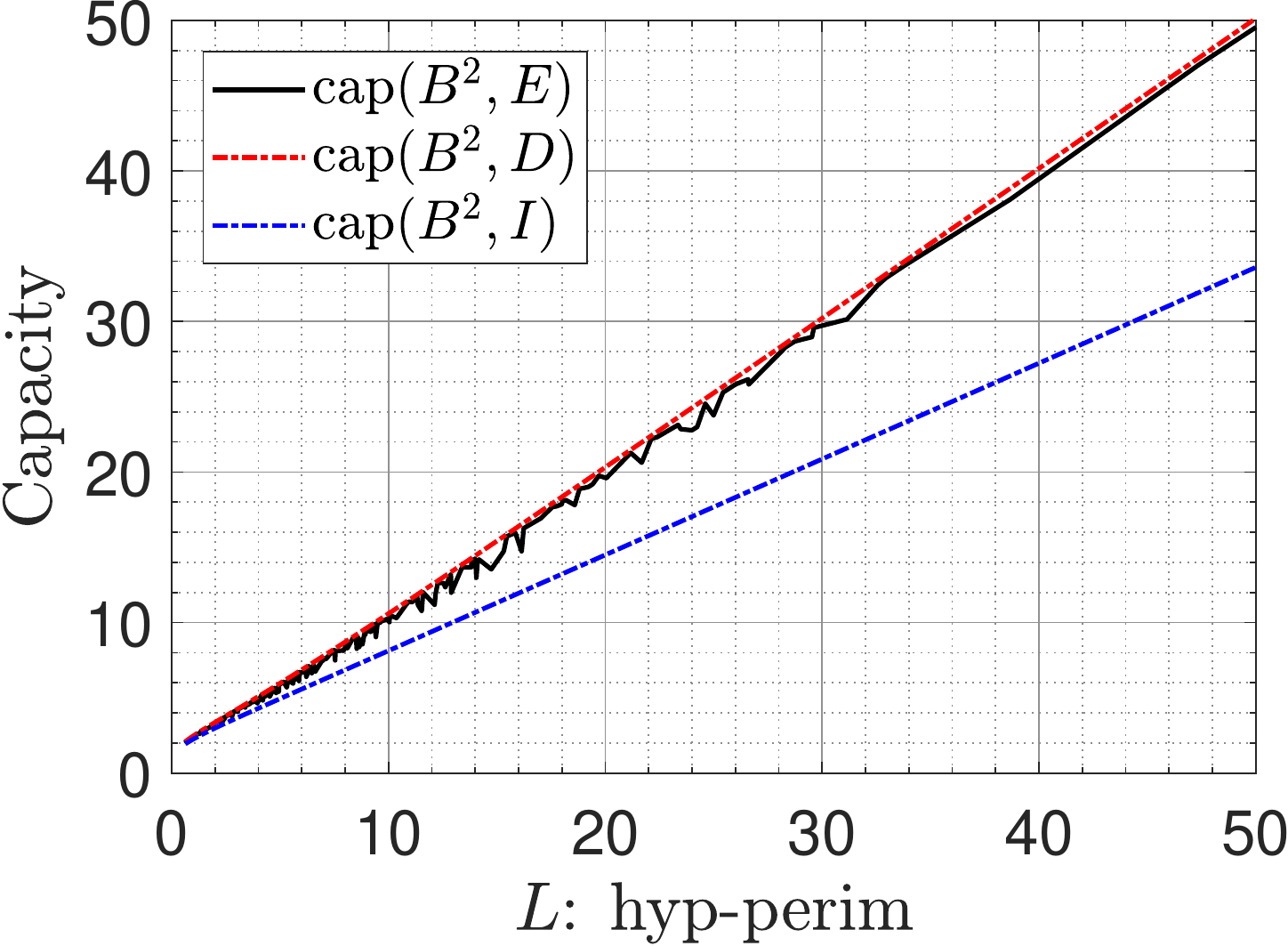}}
}
\caption{The capacities $\capa(\B^2,E)$, $\capa(\B^2,D)$, and $\capa(\B^2,I)$  for the convex Euclidean polygonal domain $E$.}
\label{fig:euc-plg-conv-res1}
\end{figure}

\end{nonsec}

\begin{nonsec}{\bf Hyperbolically convex polygons.}
Assume that $E\subset\B^2$ is a hyperbolically convex polygonal closed region such that $\partial E$ is a hyperbolic polygon with $m$ vertices $v_{1}$, $v_{2}$, $\ldots$, $v_{m}$ chosen as in~\eqref{eq:conv-thteta} (see Figure~\ref{fig:hyp-plg-conv} for $m=10$ and $s=0.9308$). 

\begin{figure}[H] %
\centerline{
\scalebox{0.6}{\includegraphics[trim=0cm 0cm 0cm 0cm,clip]{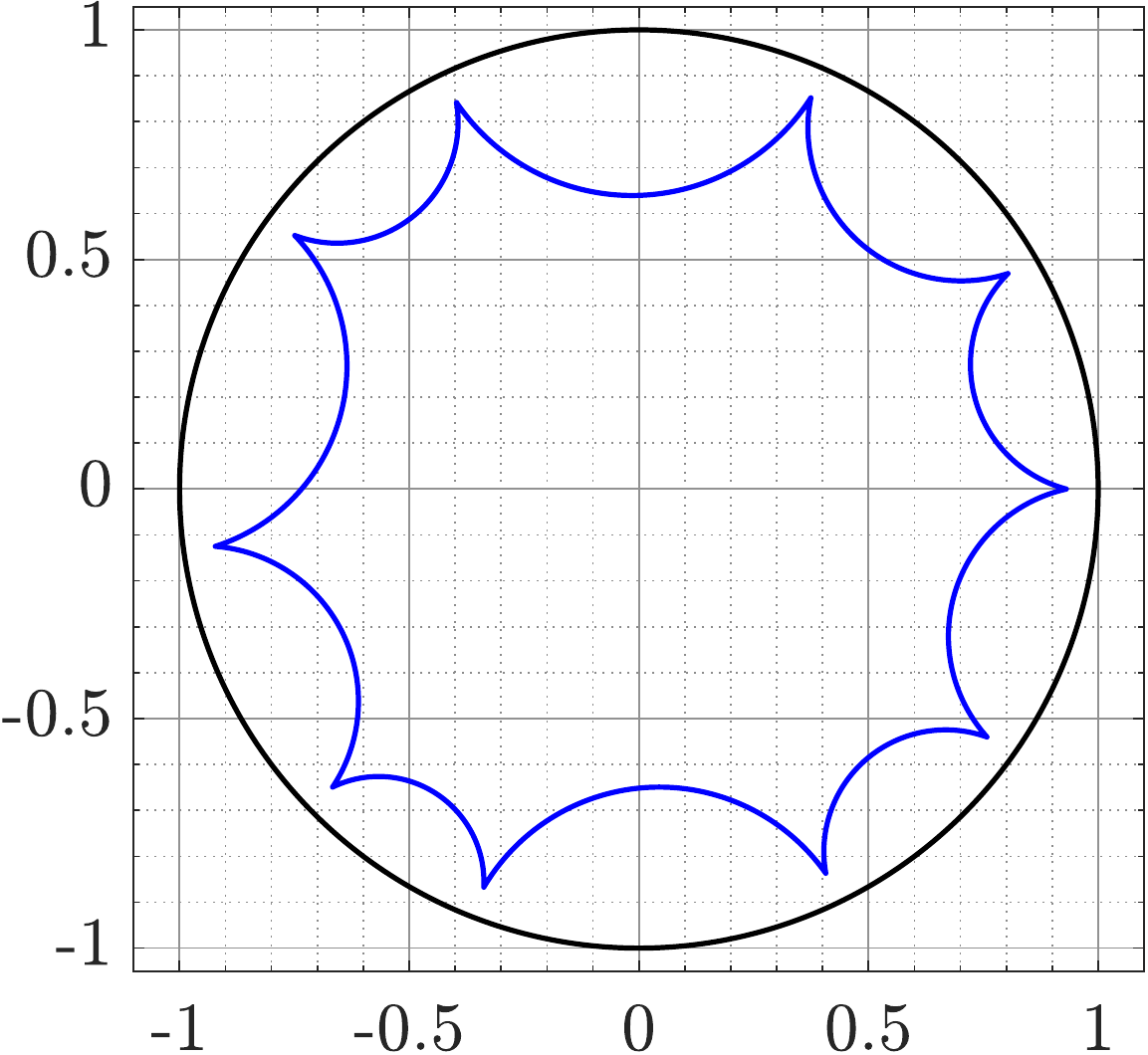}}
}
\caption{A hyperbolically convex polygonal domain for $m=10$ and $s=0.9308$.}
\label{fig:hyp-plg-conv}
\end{figure}

The values of $L=\mbox{h-perim}_{\B^2}(\partial E)$ are then computed by 
\[
L=\sum_{j=1}^{m}\rho_{\B^2}(v_{j-1},v_{j}),
\]
where $v_0=v_m$. Let $D=B^2(0,R)$ be the Euclidean disk with $\mbox{h-perim}_{\B^2}(\partial D)=L$ and let $I=[0,a]$ where $a$ is a positive real number such that $\mbox{h-perim}_{\B^2}(I)=L$. Then $\capa(\B^2,D)$ and $\capa(\B^2,I)$ are computed as in~\eqref{eq:cap-D} and~\eqref{eq:cap-I}, respectively. The capacity $\capa(\B^2,E)$ is computed using the MATLAB function \verb|annq|. 
The values of the capacities $\capa(\B^2,E)$, $\capa(\B^2,D)$, and $\capa(\B^2,I)$ versus the hyperbolic perimeter $L$ are given in Figure~\ref{fig:hyp-plg-conv-res}. These values are computed for $200$ random values of $m$ and $s$ in~\eqref{eq:conv-thteta}. The obtained results indicate that Theorem~\ref{KuhnauThm} is valid for the above described hyperbolically convex set $E$.

\begin{figure}[H] %
\centerline{
\scalebox{0.6}{\includegraphics[trim=0cm 0cm 0cm 0cm,clip]{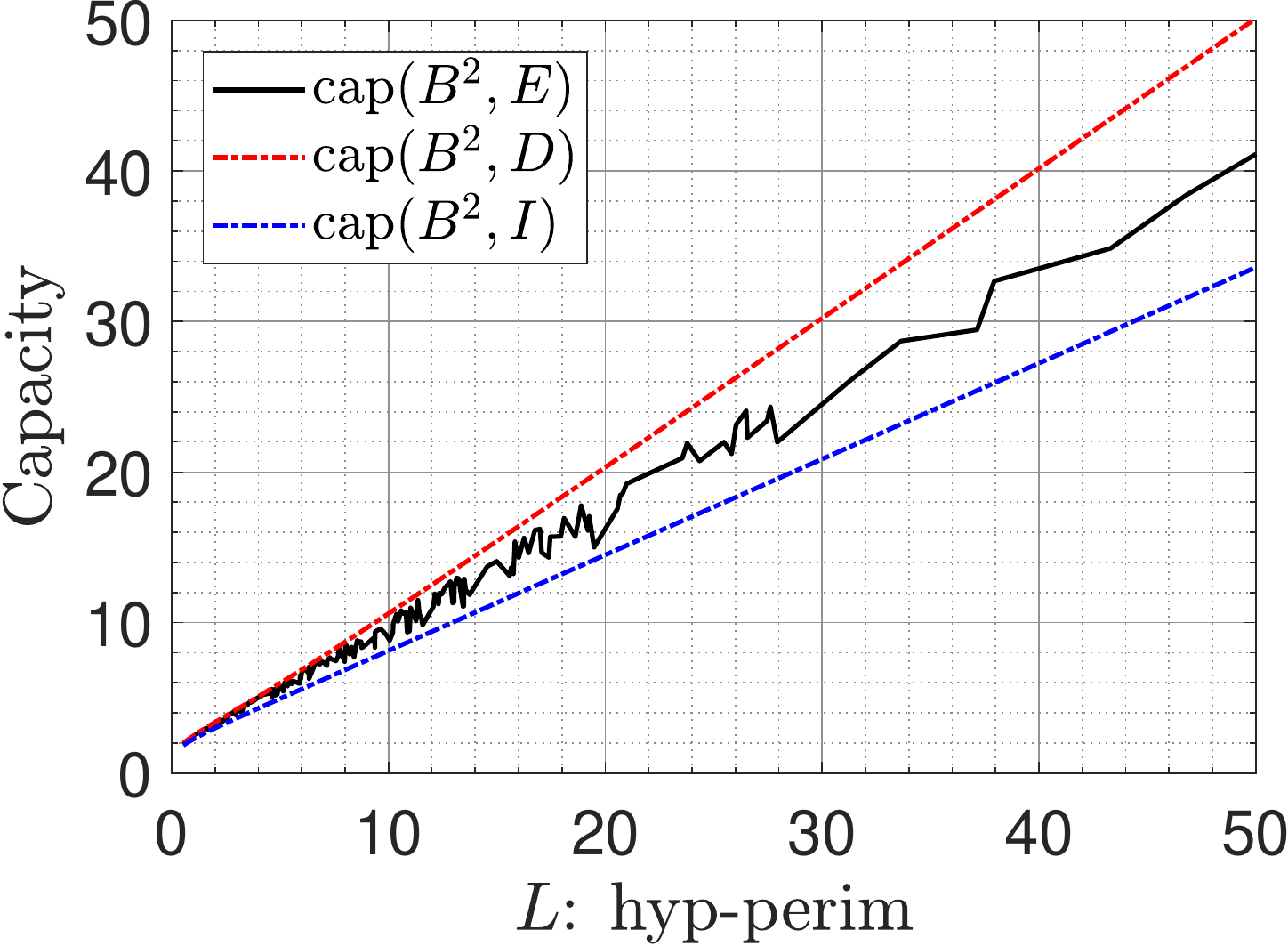}}
}
\caption{The capacities $\capa(\B^2,E)$, $\capa(\B^2,D)$, and $\capa(\B^2,I)$ for the convex hyperbolic polygonal domain $E$.}
\label{fig:hyp-plg-conv-res}
\end{figure}

\end{nonsec}

\begin{nonsec}{\bf Nonconvex Euclidean polygons.}
Assume that $E\subset\B^2$ is a nonconvex Euclidean polygonal closed region such that $\partial E$ is a polygon with $m$ vertices $v_{1}$, $v_{2}$, \ldots, $v_{m}$, where $m$ is an integer chosen randomly such that $3\le m\le 12$ and $\theta_1,\theta_2,\ldots,\theta_m$ are as in~\eqref{eq:conv-thteta}. Then, we assume that the vertices are given by
\[
v_{j} = s_j\,e^{\i \theta_j},
\]
where $s_j$ is chosen randomly such that $0.5<s_{2j-1}< 0.95$ and $0.05<s_{2j}< 0.5$ (see Figure~\ref{fig:euc-plg-nconv} for $m=12$). 

\begin{figure}[H] %
\centerline{
\scalebox{0.6}{\includegraphics[trim=0cm 0cm 0cm 0cm,clip]{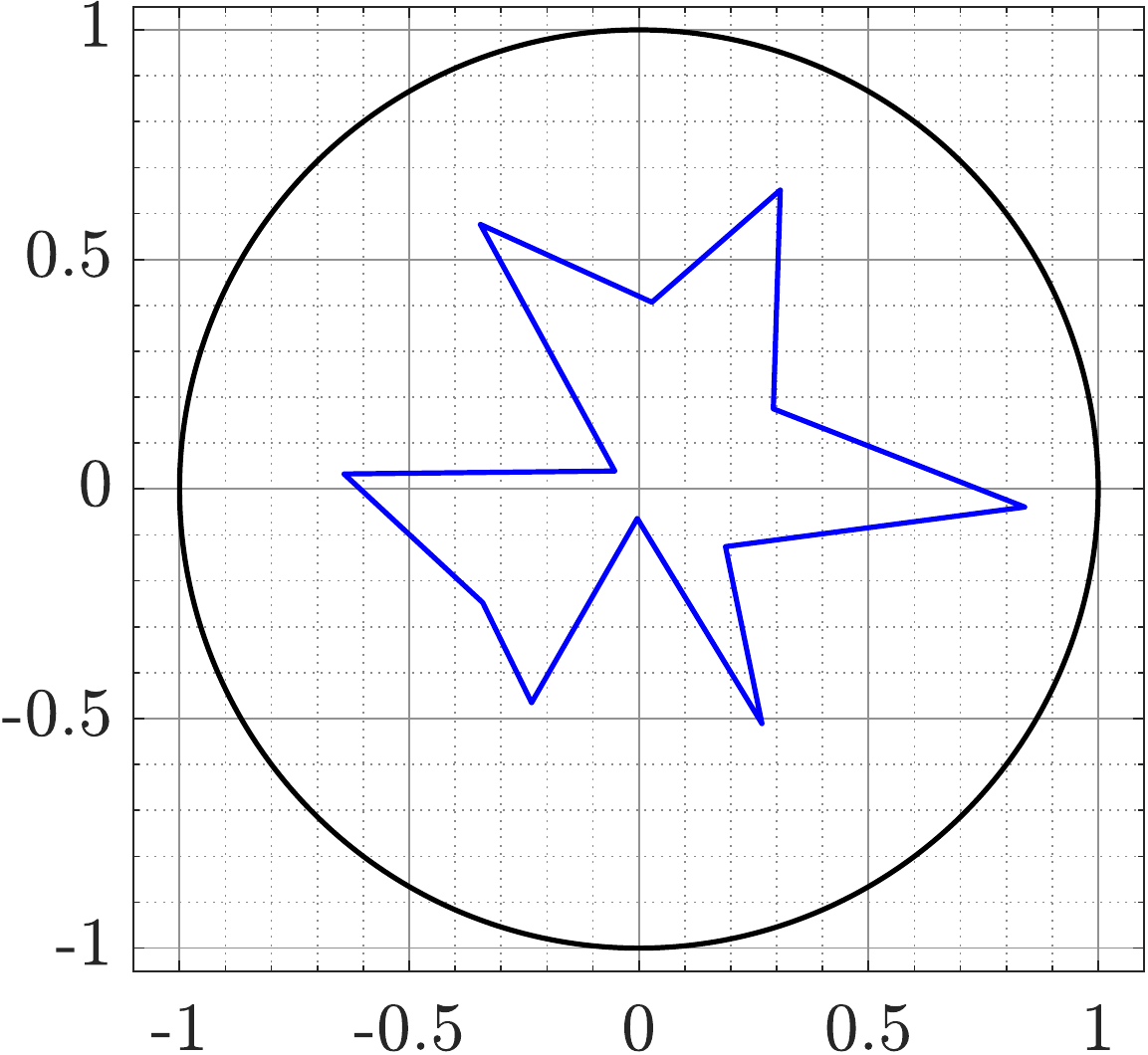}}
}
\caption{A nonconvex polygonal domain for $m=12$.}
\label{fig:euc-plg-nconv}
\end{figure}

The values of $L=\mbox{h-perim}_{\B^2}(\partial E)$ are computed by the method presented in Subsection~\ref{sec:hyp-diam}.
Let $D=B^2(0,R)$ be the Euclidean disk with $\mbox{h-perim}_{\B^2}(\partial D)=L$ and let $I=[0,a]$ where $a$ is a positive real number such that $\mbox{h-perim}_{\B^2}(I)=L$. 
The capacities $\capa(\B^2,D)$ and $\capa(\B^2,I)$ are computed as in~\eqref{eq:cap-D} and~\eqref{eq:cap-I}, respectively, and the capacity $\capa(\B^2,E)$ is computed using the MATLAB function \verb|annq|. 
The values of these capacities versus the hyperbolic perimeter $L$ are given in Figure~\ref{fig:euc-plg-conv-res}. These values are computed for $200$ random values of $m$. The obtained results show that 
Theorem~\ref{KuhnauThm} is not valid if $E$ is nonconvex.

\begin{figure}[H] %
\centerline{
\scalebox{0.6}{\includegraphics[trim=0cm 0cm 0cm 0cm,clip]{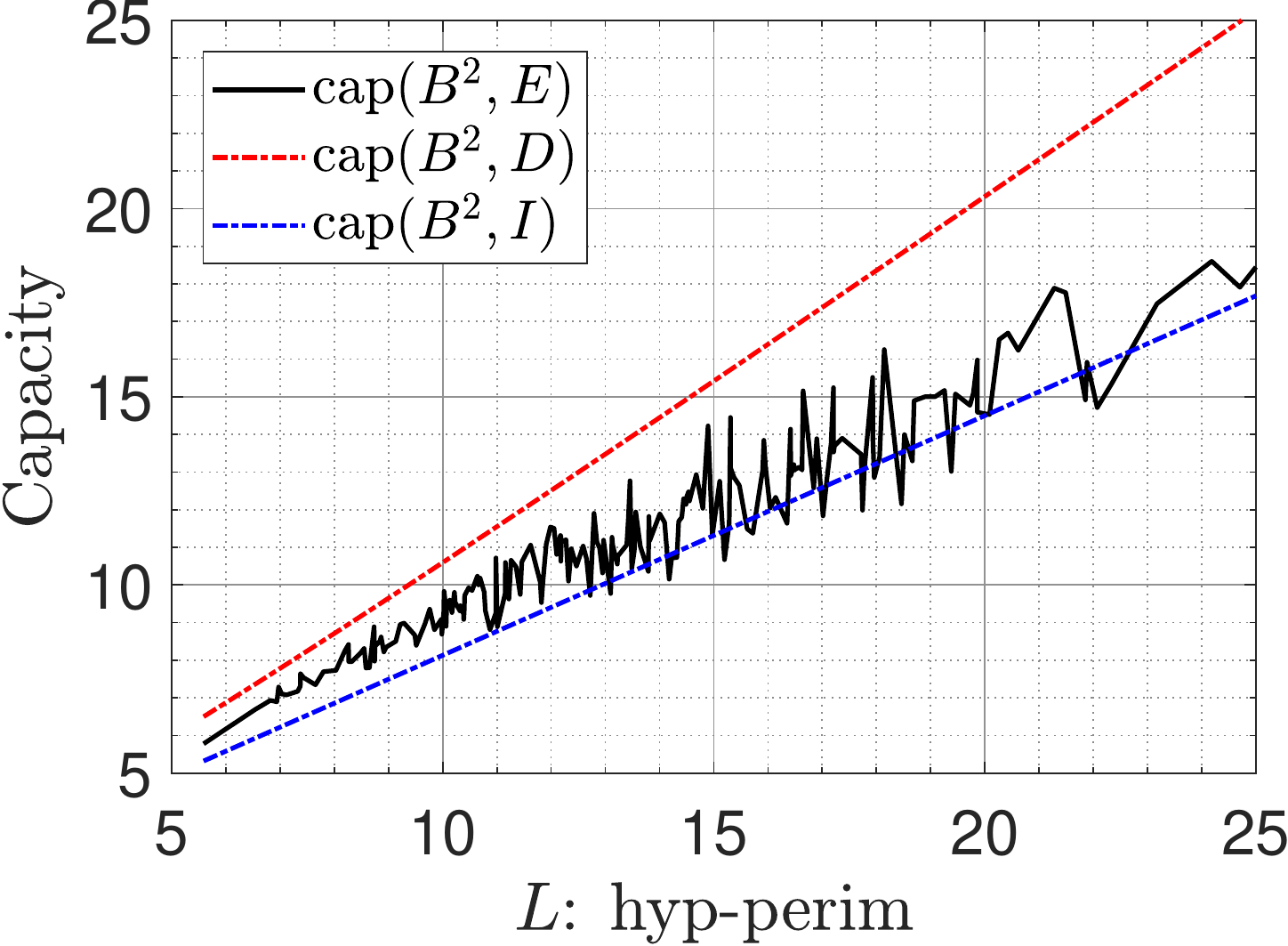}}
}
\caption{The capacities $\capa(\B^2,E)$, $\capa(\B^2,D)$, and $\capa(\B^2,I)$  for the nonconvex Euclidean polygonal domain $E$.}
\label{fig:euc-plg-conv-res}
\end{figure}

\end{nonsec}

\begin{nonsec}{\bf Nonconvex set $E(t)$.}
For a fixed $\theta\in(0,\pi\slash2)$ and a fixed $r\in(0,1)$, let $E=\{re^{si}:\theta\leq s\leq2\pi\}$. Let $t$ be so small that $\rho_{\B^2}(r, re^{\theta i})\ge3t$. We define 
\[
E(t)=\{z: \text{hypdist}(z,E)\le t\}\subset\B^2
\]
and $f(t) = \mbox{h-perim}_{\B^2}(E(t))$ (see Figure~\ref{fig:fig-E(t)}).
Since
\[
\rho_{\B^2}(r, re^{\theta i})=2\arsh\frac{2r\sin(\theta/2)}{1-r^2},
\]
we choose $t$ such that $0<3t\le2\arsh(2r\sin(\theta/2)/(1-r^2))$. 

\begin{figure}[H] %
\centerline{
\scalebox{0.6}{\includegraphics[trim=0cm 0cm 0cm 0cm,clip]{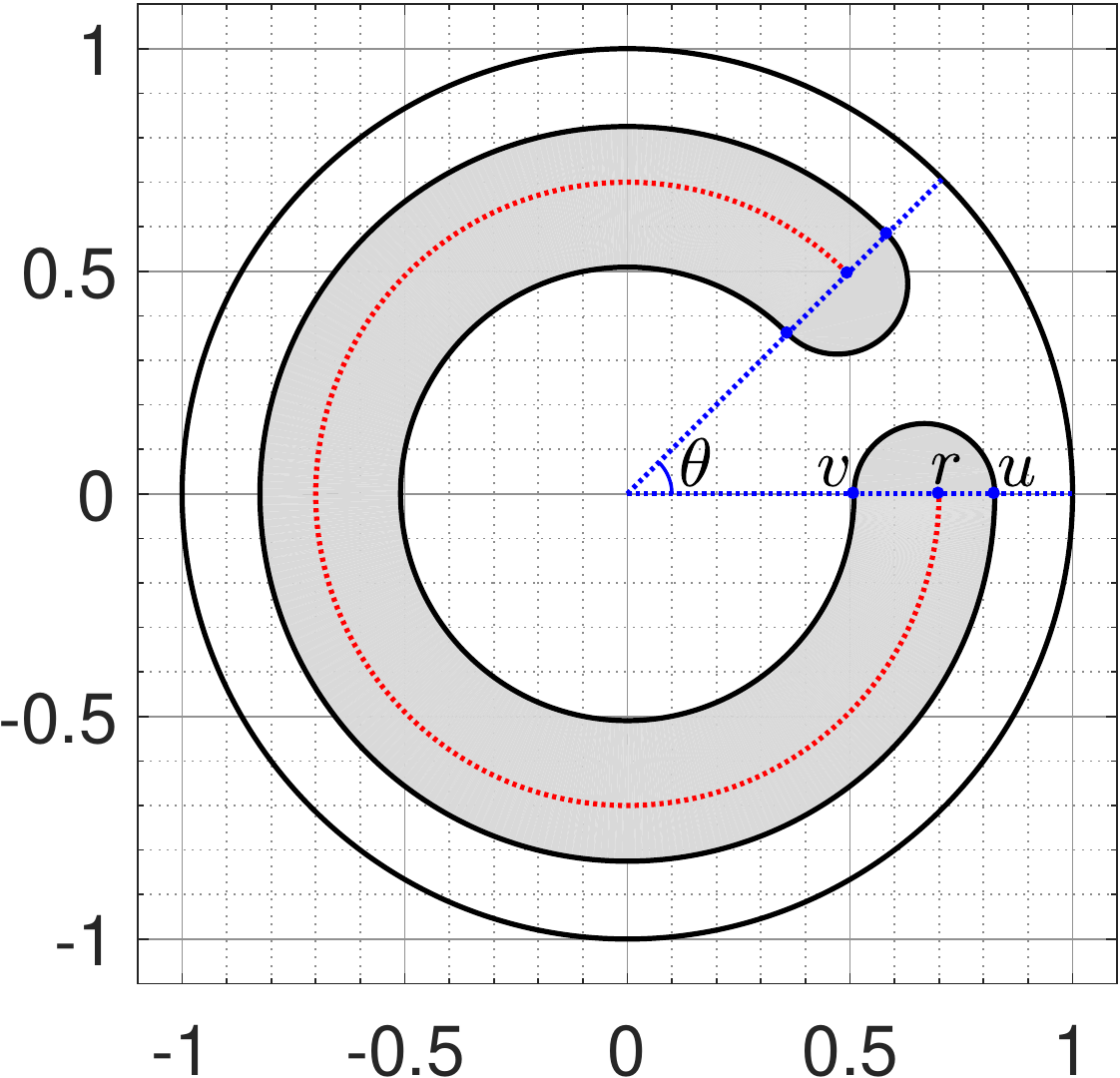}}
}
\caption{The nonconvex set $E(t)$ for $t=\frac{2}{3}\arsh(2r\sin(\theta/2)/(1-r^2))$ where $\theta=\pi/4$ and $r=0.6$.}
\label{fig:fig-E(t)}
\end{figure}

The boundary of $E(t)$ consists of two circular arcs and two hyperbolic half-circles.
Let the radius of the two circular arcs be $u$ and $v$ where $0<v<r<u$ and $\rho_{\B^2}(r,v)=\rho_{\B^2}(r,u)=t$.
Let
\[
\hat v = 2\arth(v), \quad \hat r = 2\arth(r), \quad \hat u = 2\arth(u), 
\]
then $\hat u - \hat r = \hat r - \hat v=t$, and hence
$\hat u = \hat r + t$ and $\hat v = \hat r - t$. Thus
\[
u = \th\left(\frac{\hat u}{2}\right) = \th\left(\frac{\hat r + t}{2}\right)
 = \th\left(\!\arth(r)+\frac{t}{2}\right)
\]
and
\[
v = \th\left(\frac{\hat v}{2}\right) = \th\left(\frac{\hat r - t}{2}\right)
 = \th\left(\!\arth(r)-\frac{t}{2}\right).
\]
The two circular arcs have the hyperbolic center $0$ and the hyperbolic radii $\hat u$ and $\hat v$. 
Hence, by Lemma~\ref{perimBe}, the total hyperbolic length of these two circular arcs is 
\[
\frac{2\pi-\theta}{2\pi}\times 2\pi(\sh(\hat u)+\sh(\hat v))
=(2\pi-\theta)(\sh(\hat u)+\sh(\hat v))
=(2\pi-\theta)\left(\frac{2u}{1-u^2}+\frac{2v}{1-v^2}\right),
\]
which can be simplified as
\[
(2\pi-\theta)\frac{4r\,\ch(t)}{1-r^2}.
\]
The two hyperbolic half-circles have hyperbolic centers $r,r^{\i\theta}$ and a hyperbolic radius $t$. Thus, by Lemma~\ref{perimBe}, the total hyperbolic length of these two hyperbolic half-circles is $2\pi\sh(t)$.
Hence, the hyperbolic perimeter of the set $E(t)$ is
\[
f(t) = 2\pi\sh(t)+\frac{4r(2\pi-\theta)\ch(t)}{1-r^2}.
\]

Finally, let $I(t)=[0,a(t)]$ where $a(t)$ is a positive real number such that $\mbox{h-perim}_{\B^2}(I(t))=f(t)$, i.e., $a(t) = \tth(f(t)/4)$. Thus, as in~\eqref{eq:cap-I},
\begin{equation}\label{eq:cap-I-f}
\capa(\B^2,I(t)) = \frac{2\pi}{\mu(\tth(f(t)/4))}.
\end{equation}
Now, we can compute numerically $\capa(\B^2, E(t))$ using the MATLAB function \verb|annq|.
These numerical results are presented in Figure~\ref{fig:cap-E(t)}, which shows that $\capa(\B^2,E(t))<\capa(\B^2,I(t))$ for $\theta=\pi/4$, $r=0.5$, and $0.05\le t\le\frac{2}{3}\arsh(2r\sin(\theta/2)/(1-r^2))$. For $\theta=\pi/4$ and $r=0.75$, $\capa(\B^2,E(t))<\capa(\B^2,I(t))$ for some values of $t$ and $\capa(\B^2,E(t))>\capa(\B^2,I(t))$ for other values of $t$. As in the previous example, this example shows that 
Theorem~\ref{KuhnauThm} is not applicable to nonconvex sets. 

\begin{figure}[H] %
\centerline{
\scalebox{0.6}{\includegraphics[trim=0cm 0cm 0cm 0cm,clip]{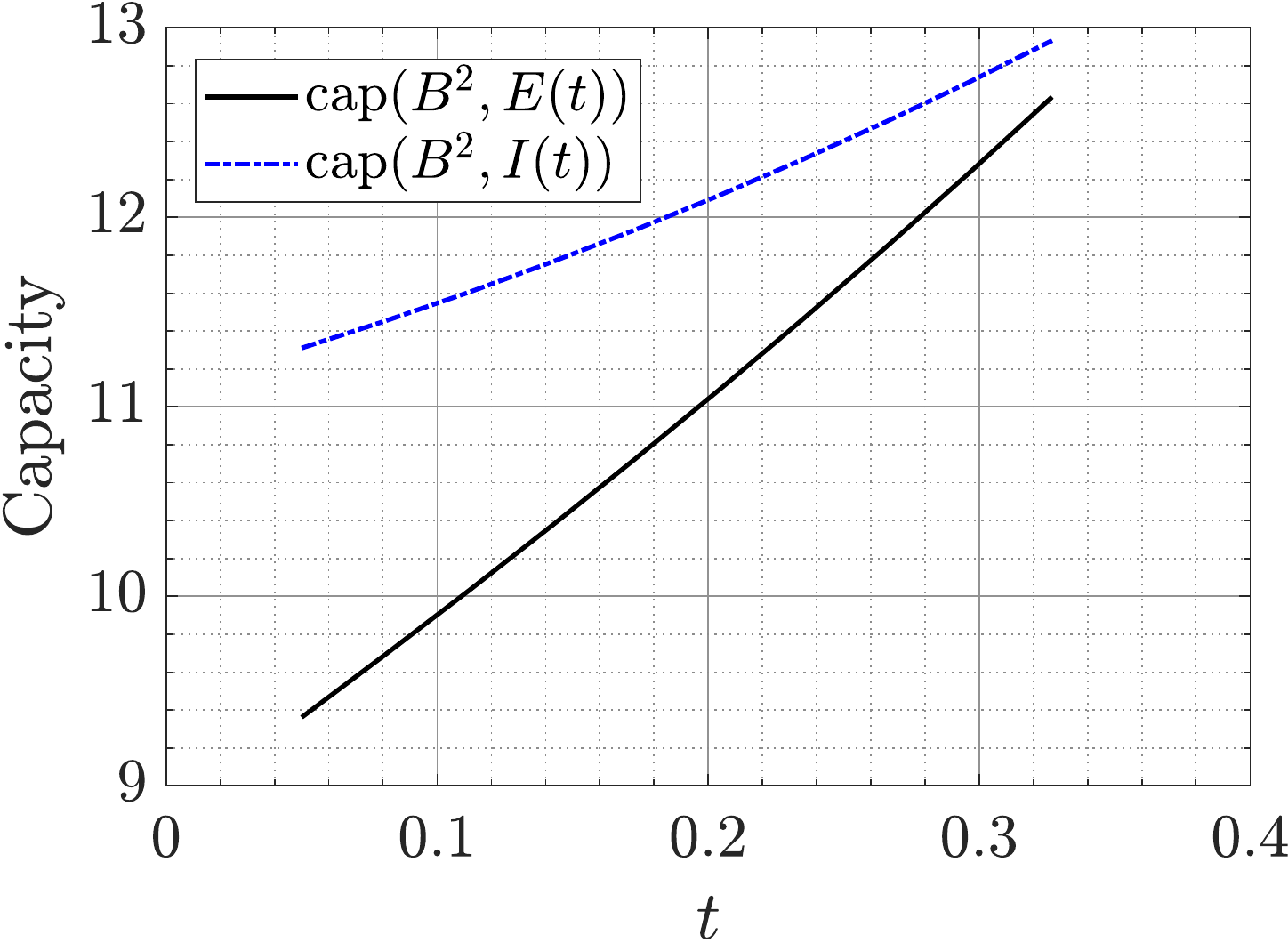}}
\hfill
\scalebox{0.6}{\includegraphics[trim=0cm 0cm 0cm 0cm,clip]{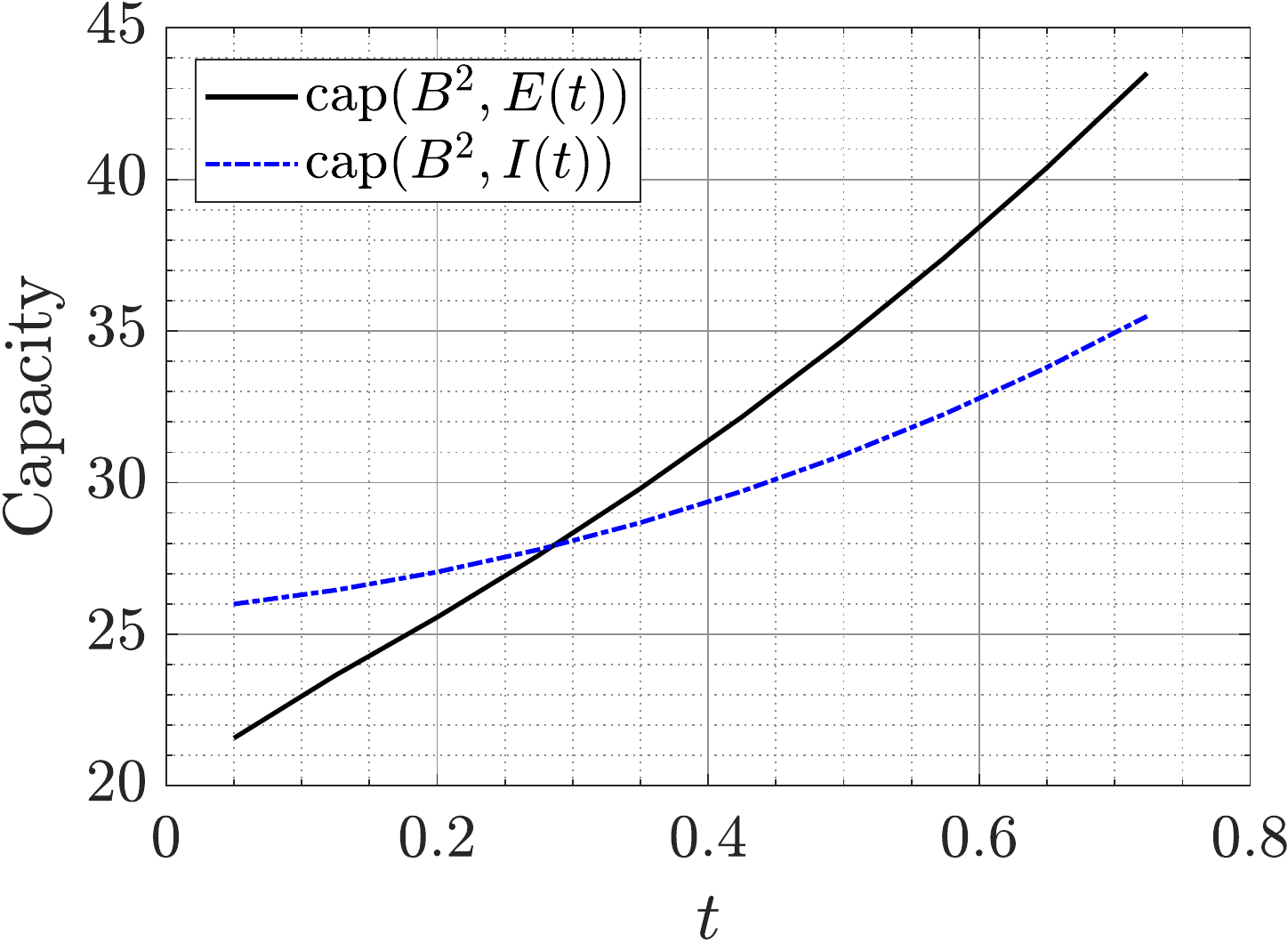}}
}
\caption{The capacity for the nonconvex set $E(t)$ for $r=0.5$ (left) and $r=0.75$ (right) where $\theta=\pi/4$ and $0.05\le t\le\frac{2}{3}\arsh(2r\sin(\theta/2)/(1-r^2))$.}
\label{fig:cap-E(t)}
\end{figure}

\end{nonsec}

\begin{nonsec}{\bf Polygon in polygon.}
Let $G$ be the simply connected domain interior to the polygon with the vertices $3\i$, $6+3\i$, $6+9\i$, $-6+9\i$, $-6-9\i$, $5-9\i$, $5-3\i$, and $-3\i$, and let $E(t)\subset G$ be the simply connected domain interior to the polygon with the vertices $-2-(6-t)\i$, $4-(6-t)\i$, $4-(6+t)\i$, $-4-(6+t)\i$, $-4+(6+t)\i$, $5+(6+t)\i$, $5+(6-t)\i$, and $-2+(6-t)\i$ for $0<t<3$ (see Figure~\ref{fig:plg-plg} for $t=1$). The values of the capacity $\capa(G, E(t))$ are calculated numerically for $0.01\le t\le 2$ using the MATLAB function \verb|annq| and the results are presented in Figure~\ref{fig:cap-plg-plg}. 


The hyperbolic perimeter of $E(t)$ with respect to the metric $\rho_G$ will be computed using the method described in Subsection~\ref{sec:hyp-diam}. Let $f(t)=\mbox{h-perim}_G(E(t))$ and let $D(t)=B^2(0,R(t))$ be the Euclidean disk with $\mbox{h-perim}_{\B^2}(\partial D(t))=f(t)$. Then, it follows from~\eqref{eq:disk-cap-P} that $\capa(\B^2, D(t))$ is given by
\[
\capa(\B^2,D(t)) = \frac{2\pi}{\log\left(\sqrt{1+(2\pi/f(t))^2}+2\pi/f(t)\right)}.
\]
The values of the capacities $\capa(G, E(t))$ and  $\capa(\B^2, D(t))$ for $0.01\le t\le 2$ are presented in Figure~\ref{fig:cap-plg-plg}. 
The polygonal domain here is of the type considered in Subsection~\ref{sec:trapezium} with $m=8$. Thus, a lower bound for the capacity $\capa(G,E(t))$ can be obtained from the inequality~\eqref{eq:cap-Qj}, where the quadrilaterals $Q_j$ are defined as in Subsection~\ref{sec:trapezium}. The computed values for this lower bound are presented in Figure~\ref{fig:cap-plg-plg}. As we can see from these results, the values of this lower bound become close to the values of $\capa(G,E(t))$ as $t$ increases.

\begin{figure}[H] %
\centerline{
\hfill
\scalebox{0.6}{\includegraphics[trim=0cm 0cm 0cm 0cm,clip]{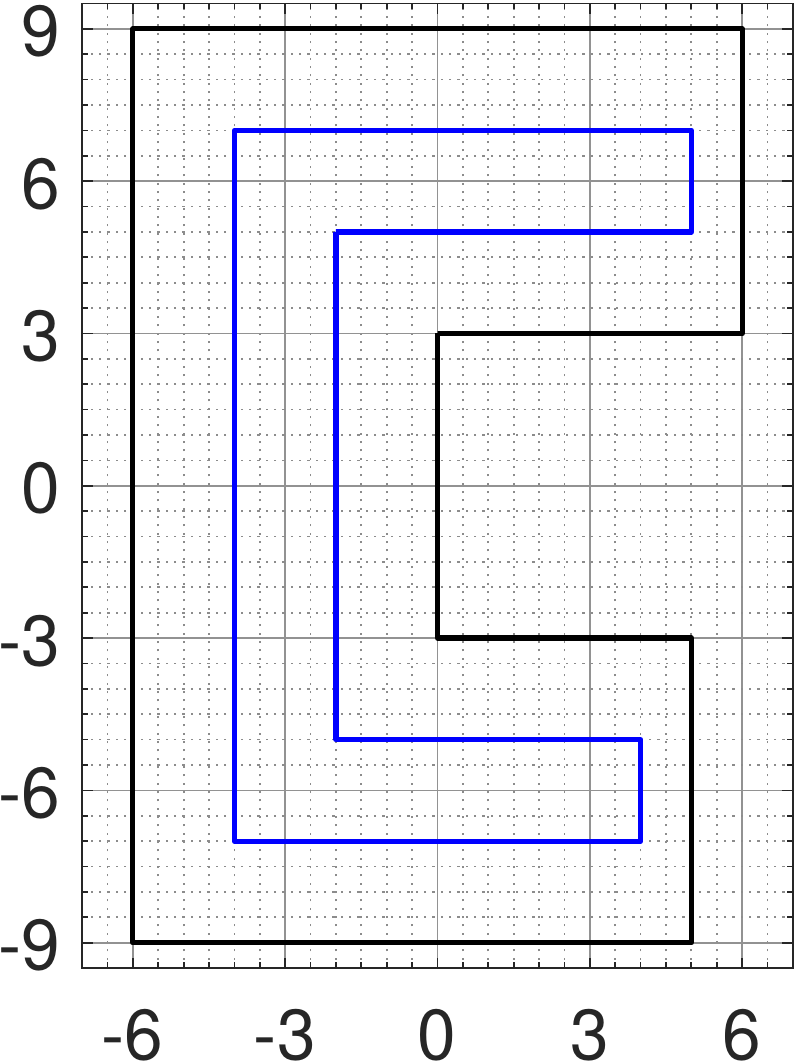}}
\hfill
\scalebox{0.6}{\includegraphics[trim=0cm 0cm 0cm 0cm,clip]{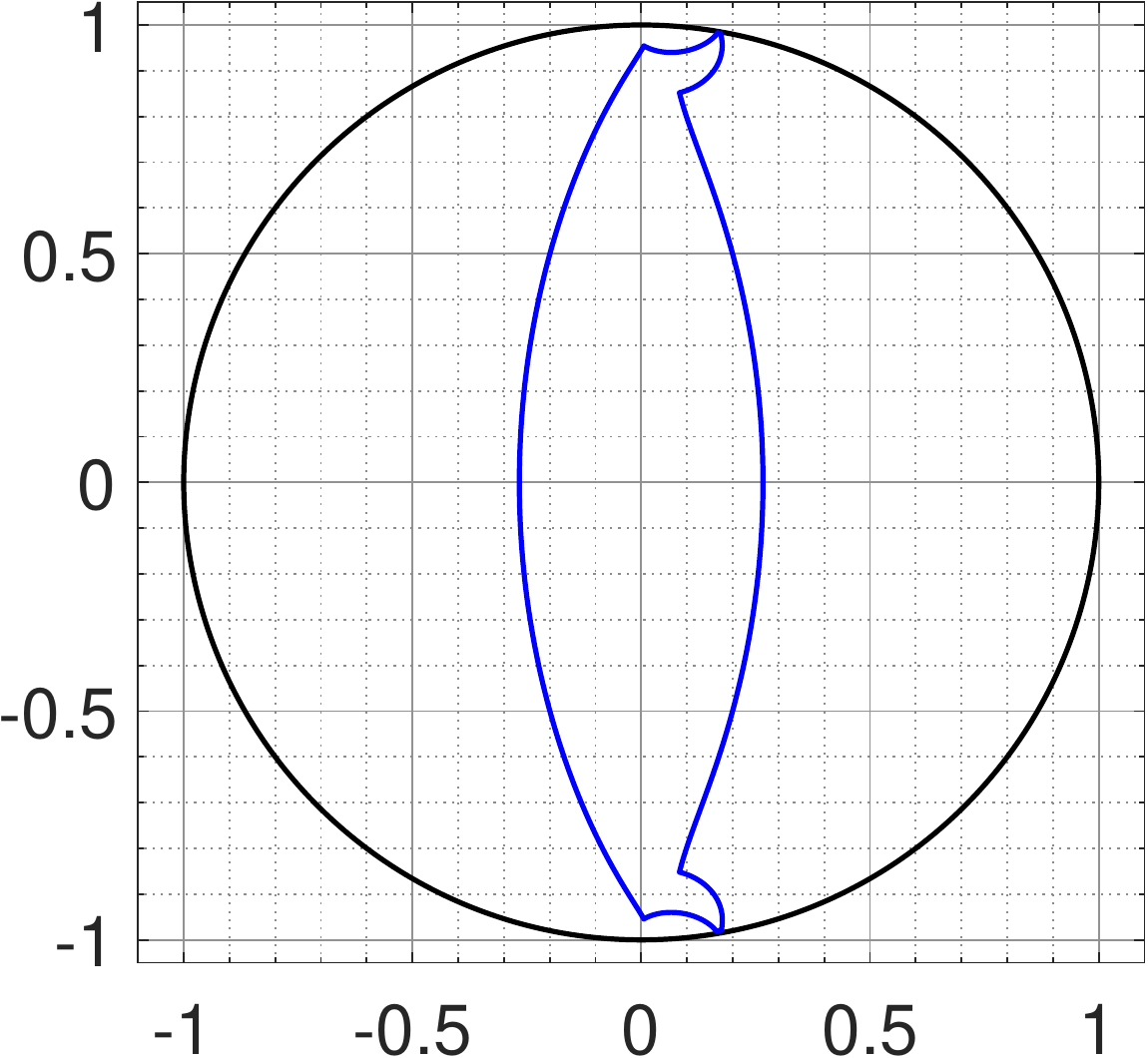}}
\hfill
}
\caption{On the left, the simply connected polygonal domain $E(t)$ inside the simply connected domain $G$ for $t=1$. On the right, the domain $G$ is mapped onto the unit disk and $E(t)$ is mapped onto a simply connected domain $\hat E(t)$ inside the unit disk.}
\label{fig:plg-plg}
\end{figure}

\begin{figure}[H] %
\centerline{
\scalebox{0.6}{\includegraphics[trim=0cm 0cm 0cm 0cm,clip]{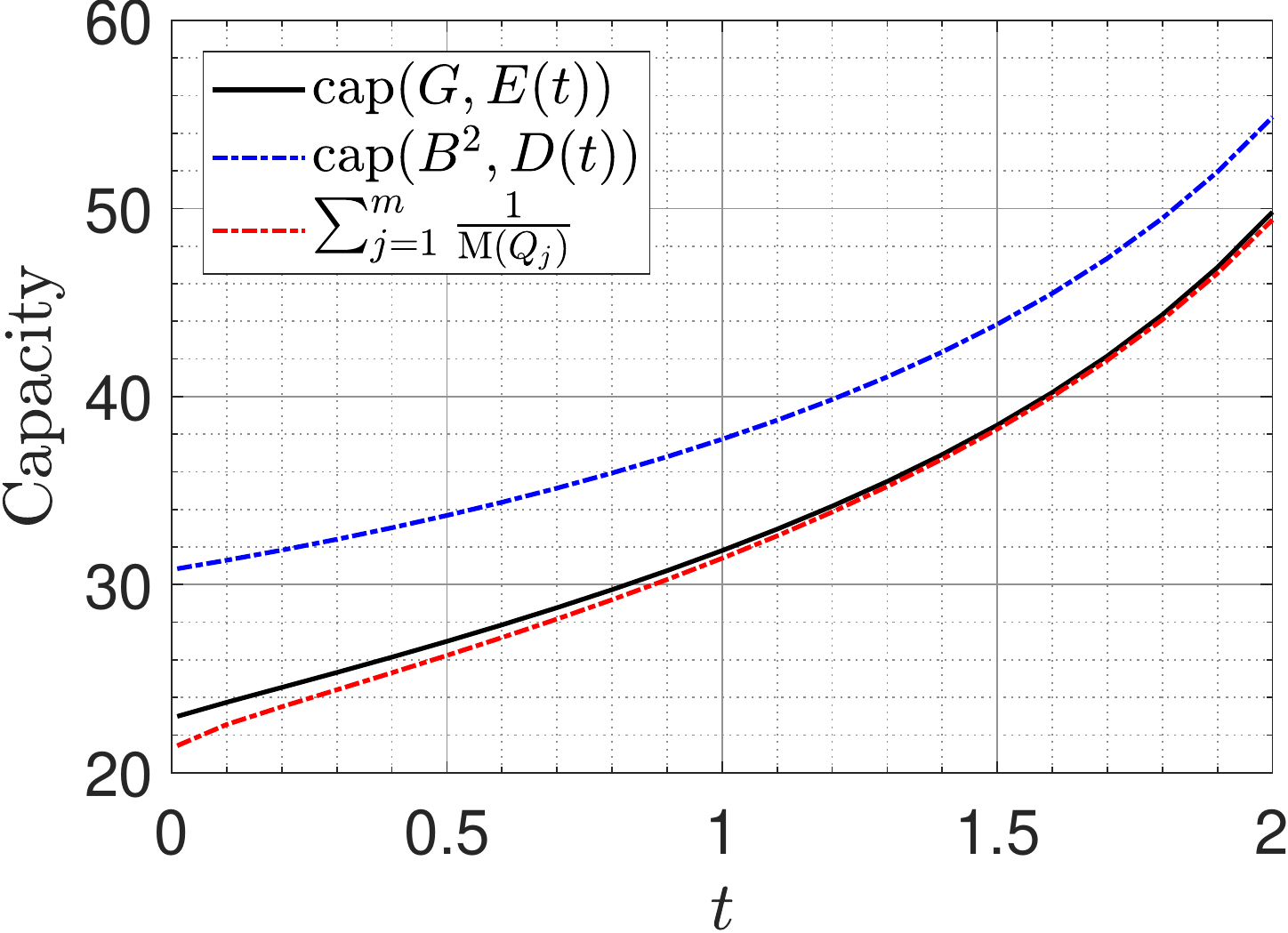}}
}
\caption{The values of the capacities $\capa(\B^2, E(t))$ with the lower and upper bounds.}
\label{fig:cap-plg-plg}
\end{figure}

\end{nonsec}

\begin{nonsec}{\bf Capacity of a half disk.}
We consider here the capacity
\begin{equation}\label{eq:halfdisk}
\capa(\B^2,E),\quad E=\overline{B}_\rho(x,t)\cap\{z\,:\,{\rm Im}z\geq0\},
\end{equation}
where $0<x<1$ and $t>0$ (see Figure~\ref{fig:halfdisk}). By Lemma \ref{perimBe}, the hyperbolic perimeter of $E$ is 
\begin{align*}
P\equiv\pi\,{\rm sh}t+2t.     
\end{align*}
We now give upper and lower bounds for $\capa(\B^2,E)$ using Theorems \ref{KuhnauThm} and~\ref{fwg71}, and Lemma~\ref{lem_capgamma}. Lemma \ref{lem_capgamma} based on
symmetrization yields
\begin{align}\label{boundFromcapgamma}
\capa(\B^2,E)\geq\gamma_2 ({1}/{{\rm th}t})=\frac{2\pi}{\mu({\rm th}t)},    
\end{align}
Theorem \ref{KuhnauThm} yields
\begin{align}\label{boundFromKuhnau}
\capa(\B^2,E)\geq\frac{2\pi}{\mu({\rm th}s)},\quad s={P}/{4} =(\pi\,{\rm sh}t+2t)/4,    
\end{align}
and Theorem \ref{fwg71} with Proposition \ref{capPerim}  yield
\begin{align}\label{boundFromfwg71}
\capa(\B^2,E)\leq\frac{2\pi}{\log((2\pi+\sqrt{P^2+(2\pi)^2})\slash P)}.   
\end{align}
The values of the capacity $\capa(\B^2,E)$ are calculated numerically for several values of $t$ using the MATLAB function \verb|annq|. The obtained results alongside the values of the bounds in~\eqref{boundFromcapgamma}, \eqref{boundFromKuhnau}, and~\eqref{boundFromfwg71} are given in Table~\ref{tab:halfdisk}. The values of the capacity $\capa(\B^2,E)$ should not depend on the values $x$ and this fact could be used to check the accuracy of the MATLAB function \verb|annq| used to compute $\capa(\B^2,E)$. It is clear from Table~\ref{tab:halfdisk} that the obtained values of $\capa(\B^2,E)$ for $x=0.5$ and $x=0.75$ are almost identical, where the absolute value of the differences between these values are: $1.0\times10^{-12}$ for $t=0.5$,  $1.1\times10^{-11}$ for $t=1$,  $2.1\times10^{-10}$ for $t=2$, and $1.8\times10^{-9}$ for $t=3$. For $x=0.75$, the boundary of the inner half circle becomes even closer to the outer boundary compared to the case $x=0.5$ especially for large $t$ (see Figure~\ref{fig:halfdisk}). Hence, the results obtained for $x=0.75$ will not be as accurate as for $x=0.5$ and this could explain the increase in the absolute value of the differences between the obtained values of $\capa(\B^2,E)$ for $x=0.5$ and $x=0.75$.

\begin{figure}[H] %
\centerline{
\hfill
\scalebox{0.6}{\includegraphics[trim=0cm 0cm 0cm 0cm,clip]{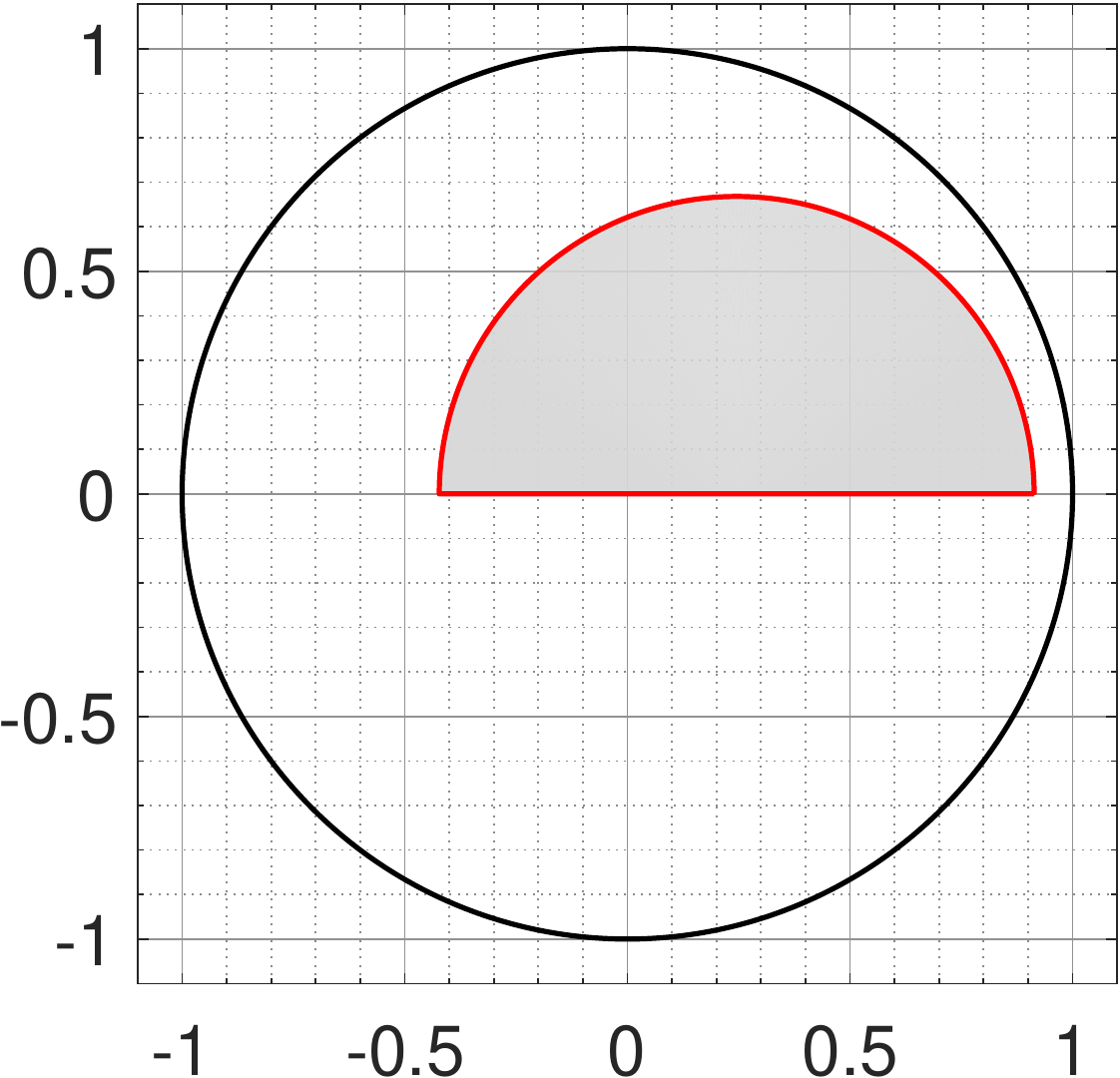}}
\hfill
\scalebox{0.6}{\includegraphics[trim=0cm 0cm 0cm 0cm,clip]{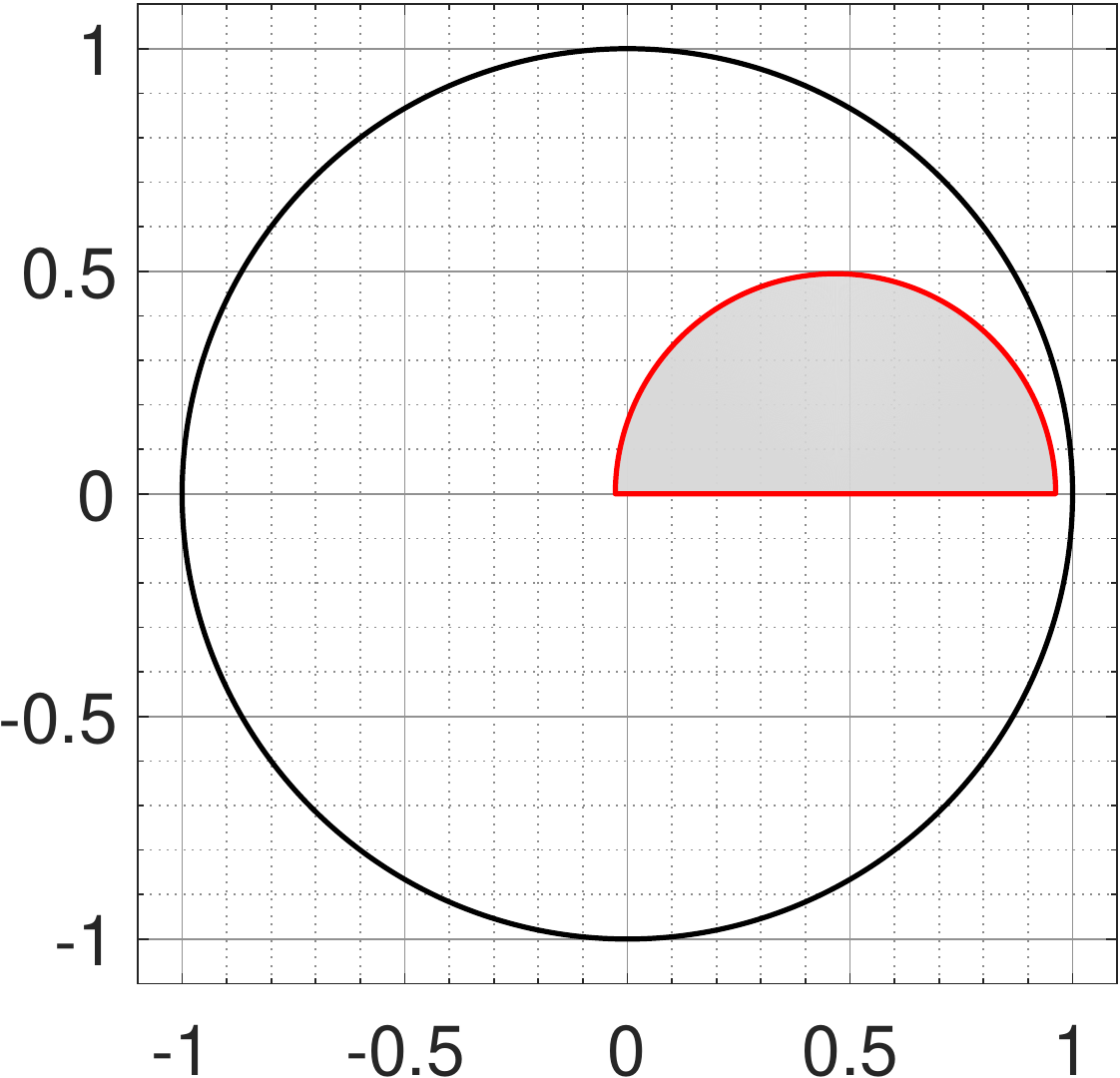}}
\hfill
}
\caption{The set $E$ in~\eqref{eq:halfdisk} for $t=2$, $x=0.5$ (left) and $t=2$, $x=0.75$ (right).}
\label{fig:halfdisk}
\end{figure}

\begin{table}[hbt]
\caption{The values of the capacity $\capa(\B^2,E)$ and the values of the bounds in~\eqref{boundFromcapgamma}, \eqref{boundFromKuhnau}, and~\eqref{boundFromfwg71} for several values of $t$.}
\label{tab:halfdisk}%
\begin{tabular}{|c|c|c|c|c|c|} 
 \hline
 $t$ & L. bound~\eqref{boundFromcapgamma} & L. bound~\eqref{boundFromKuhnau} & \multicolumn{2}{c|}{$\capa(\B^2,E)$} & U. bound~\eqref{boundFromfwg71}\\ 
 \cline{4-5}
     &   & & $x=0.5$ & $x=0.75$ & \\ 
		\hline
0.5& 2.992668693658 & 3.420421711458 & 3.786736098104 & 3.786736098103 & 3.920276955667 \\
 \hline
 1 & 4.305689987396 & 5.387651654447 & 6.295457868908 & 6.295457868897 & 6.593459117182 \\
 \hline
 2 & 6.857936184536 & 11.56528464073 & 15.04612692249 & 15.04612692228 & 15.80292688543 \\ 
 \hline
 3 & 9.404520113864 & 25.62055328755 & 36.30939061754 & 36.30939061578 & 37.64629373665 \\ 
 \hline
\end{tabular}
\end{table}
	
As Table \ref{tab:halfdisk} indicates, the upper bound for the capacity is more accurate
than the lower bound in Theorem \ref{KuhnauThm} which is sharp when the set $E$ is a segment.
Thus it is natural to look for a better lower bound for a massive set such as the half disk.
The next lemma provides such a bound (see Table~\ref{tab:LB-new} and Figure~\ref{fig:newLow}).

\begin{lemma}\label{NewLowBd} Let $E$ be the set in \eqref{eq:halfdisk} for $t>0.$ Then
\begin{equation}\label{eq:LB-new}
\capa(\B^2,E) \ge \pi/\log(1/{\rm th}(t/2)) + \pi/\mu({\rm th} t) \,.
\end{equation}
\end{lemma}

\begin{proof} Without loss of generality, we may assume that $x=0\,.$ Then by \eqref{hypdisk},
$B_{\rho}(0,t)= B^2(0,{\rm th}(t/2) )\,.$ Let ${\mathbb H}^2= \{z \in \R^2: {\rm Im} z >0 \}$
and
\[
\Gamma= \Delta(E, \partial \B^2; \B^2), \quad \Gamma_1=\{\gamma \in \Gamma: \gamma \subset {\mathbb H}^2 \}, \quad  \Gamma_2=\{\gamma \in \Gamma: \gamma \subset  \B^2 \setminus {\overline {\mathbb H}}^2 \}.
\]
Then $\Gamma_1$ and $\Gamma_2$ are separate subfamilies of $\Gamma$ and hence by
a symmetry property of the modulus \cite[p. 127, Thm 4.3.3]{HKV}
\[
\capa(\B^2, E)\ge \M(\Gamma_1)+\M(\Gamma_2)= \pi/\log(1/{\rm th}(t/2)) + \pi/\mu({\rm th} t) ,
\]
and therefore the lemma is proved.
\end{proof}

\begin{table}[H]
 \caption{The lower bound~\eqref{eq:LB-new} is better than the lower bounds in~\eqref{boundFromcapgamma} and~\eqref{boundFromKuhnau}.}\label{tab:LB-new}  
   \begin{tabular}{|c|c|c|} 
 \hline
 $t$ & L. bound~\eqref{eq:LB-new} & $\capa(\B^2,E)$\\ 
 \hline
0.5&  3.72943616629721  &3.78673609810401 \\  \hline
 1 &  6.22259852750174  &6.29545786890825 \\  \hline
 2 &  14.9644594228477  &15.0461269224867 \\  \hline
 3 &  36.226461975872   &36.3093906175404 \\  \hline
\end{tabular}
\end{table}

\begin{figure}[H] %
\centerline{
\hfill
\scalebox{0.6}{\includegraphics[trim=0cm 0cm 0cm 0cm,clip]{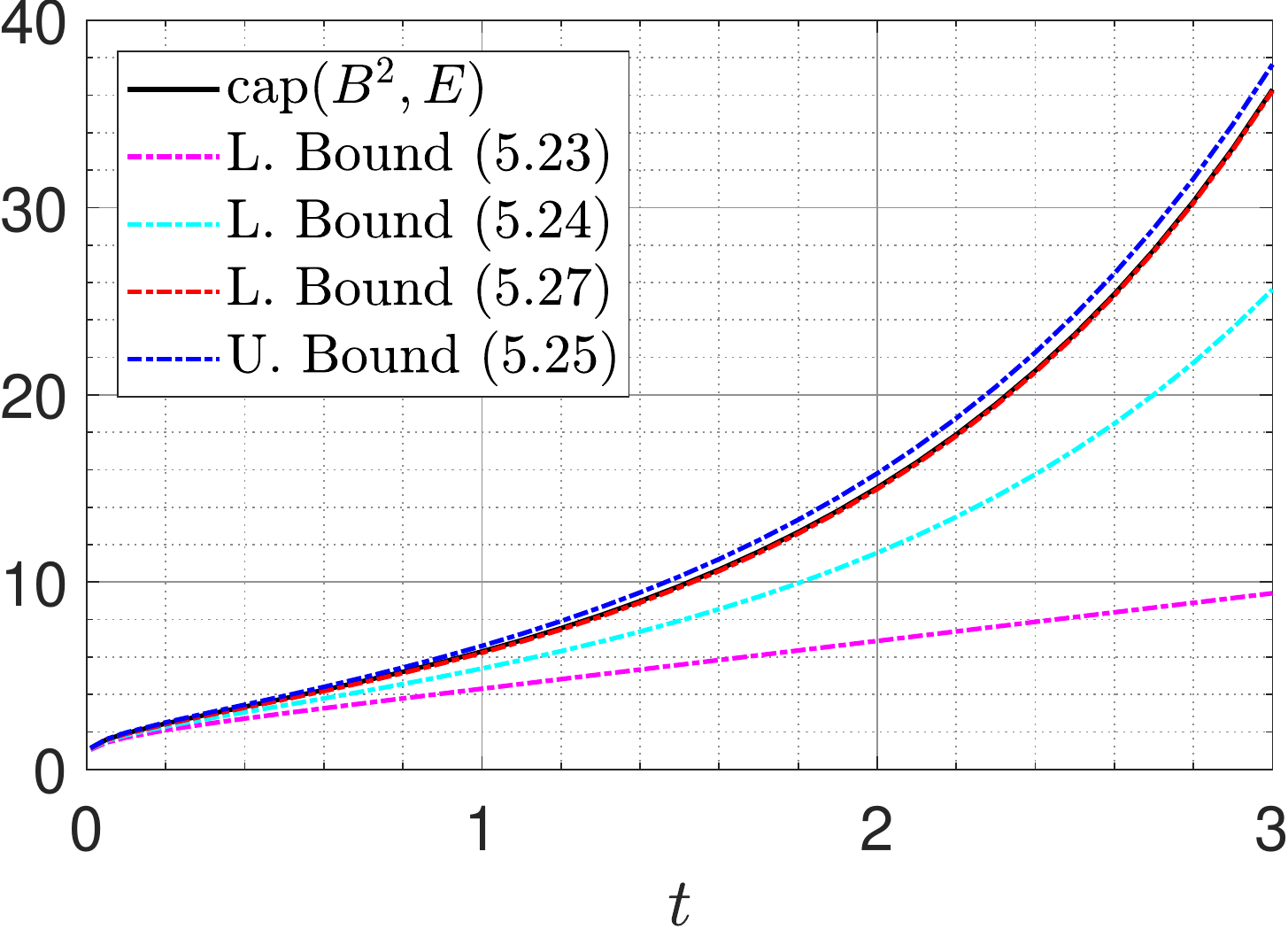}}
\hfill
\scalebox{0.6}{\includegraphics[trim=0cm 0cm 0cm 0cm,clip]{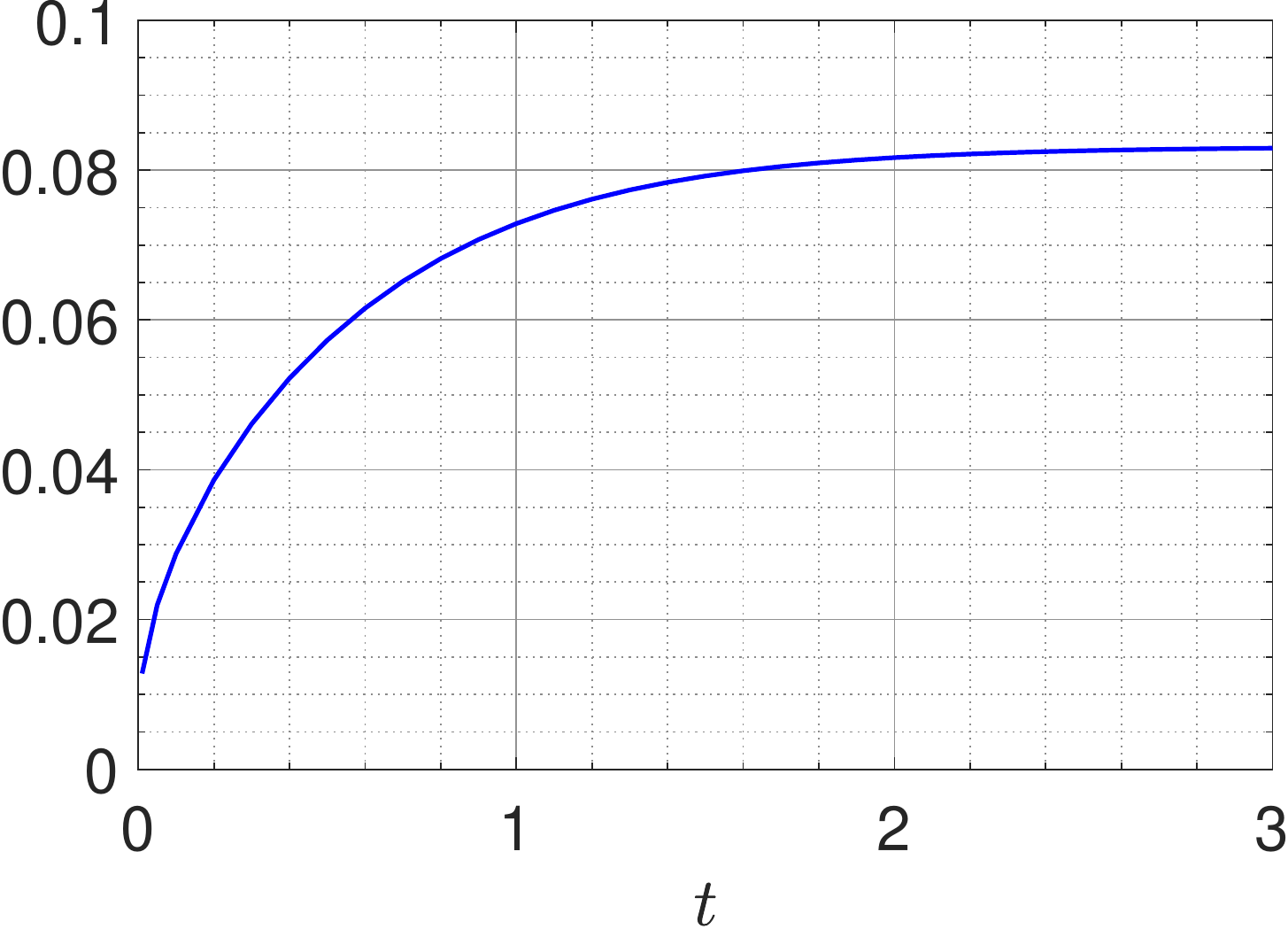}}
\hfill
}
\caption{The capacity $\capa(\B^2,E)$, the upper bound~\eqref{boundFromfwg71}, and the three lower bounds \eqref{boundFromcapgamma}, \eqref{boundFromKuhnau}, \eqref{eq:LB-new} (left) and the difference between $\capa(\B^2,E)$ and the best lower bound~\eqref{eq:LB-new} (right).}
\label{fig:newLow}
\end{figure}

\end{nonsec}

\begin{nonsec}{\bf Capacity of a rectangle with a segment.}
We consider here the capacity
\begin{equation}\label{eq:rec-seg}
\capa(F,E),\quad F=(-a,a)\times(0,b), \quad E=[\i c, \i d],
\end{equation}
where $a,b,c,d$ are positive real numbers such that $0<c<d<b$ (see Figure~\ref{fig:rec-slit} (left)).

\begin{figure}[H] %
\centerline{
\scalebox{0.55}{\includegraphics[trim=0cm 0cm 0cm 0cm,clip]{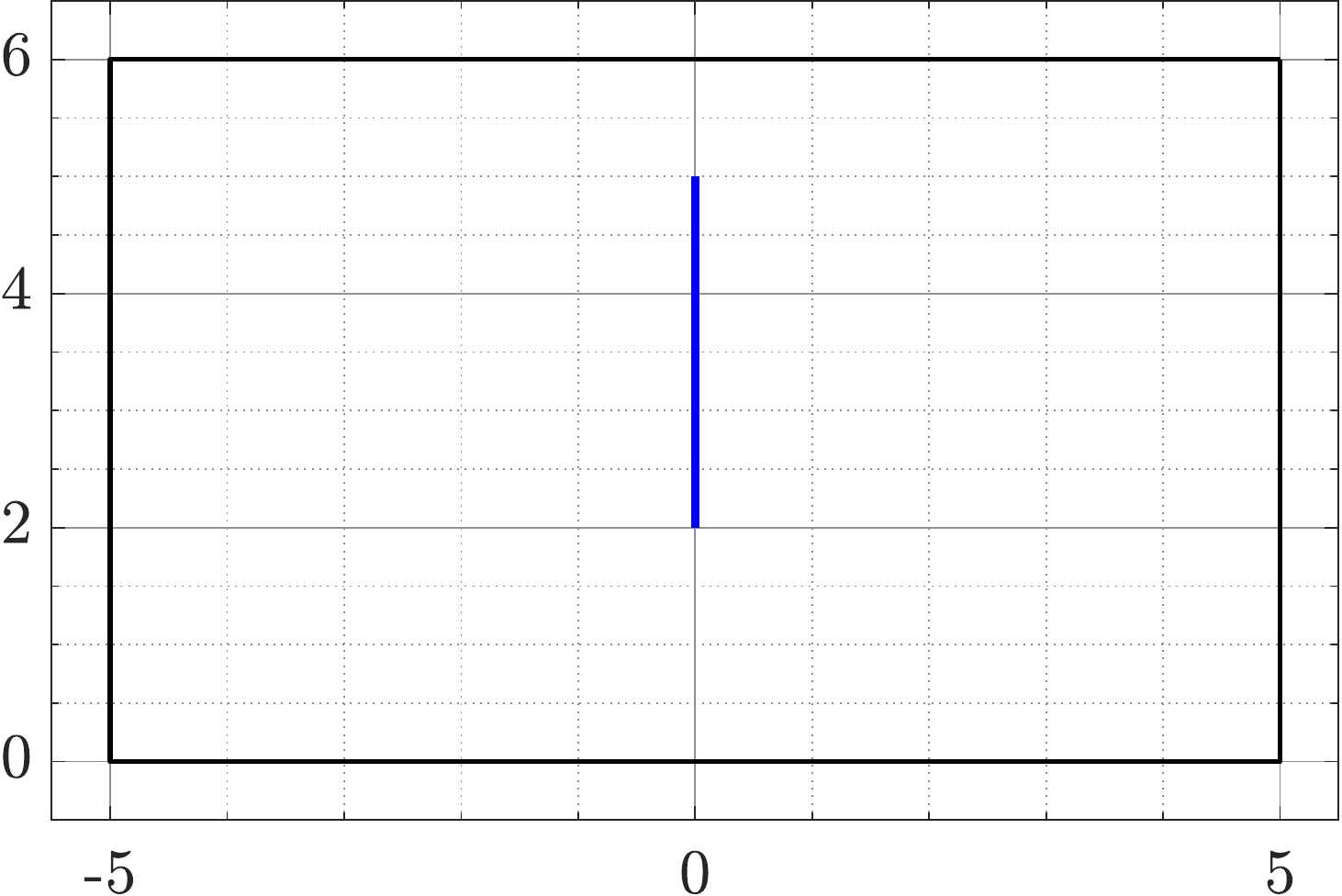}}
\hfill
\scalebox{0.55}{\includegraphics[trim=0cm 0cm 0cm 0cm,clip]{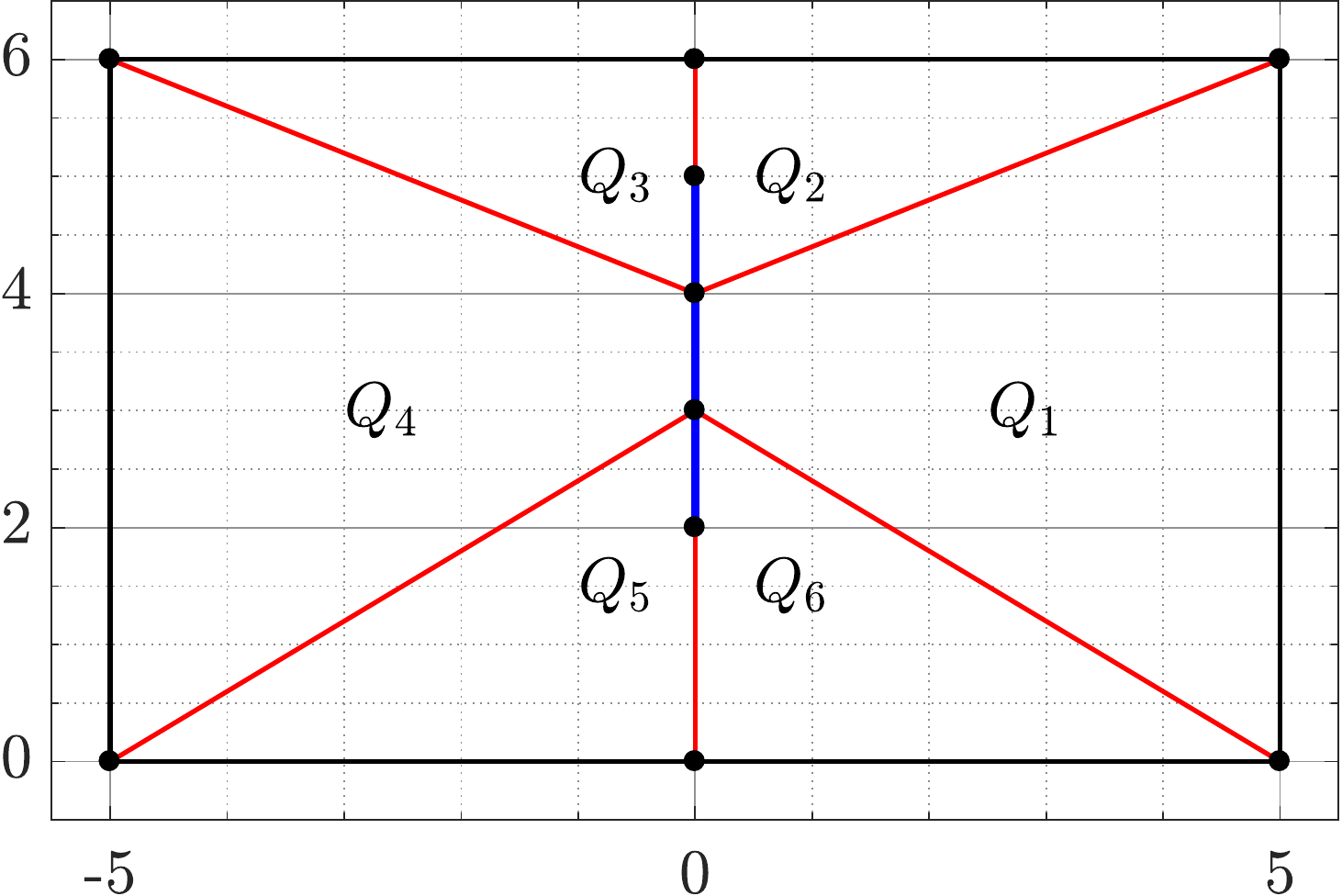}}
}
\caption{The rectangle with a segment condenser for $a=5$, $b=6$, $c=2$ and $d=5$.}
\label{fig:rec-slit}
\end{figure}

The exact value of the capacity $\capa(F,E)$ can be obtained with the help of conformal mappings. 
Let
\[
\tau=\frac{\i b}{a}, \quad q=e^{\pi\tau\i}=e^{-\pi b/a}>0, \quad k=\left(\frac{\theta_2(0,q)}{\theta_3(0,q)}\right)^2, \quad \alpha=\frac{a}{\K(k)},
\] 
where $\K(\cdot)$ is defined by~\eqref{capGro} and
\[
\theta_2(0,q)=2 \sum_{n=0}^\infty q ^{(n+1/2)^2}\,, \quad
\theta_3(0,q)=1 + 2 \sum_{n=1}^\infty q ^{n^2}\,\,.
\]
Then, it follows from~\cite[pp.~172-173]{kob} that the mapping function 
\[
z \mapsto \sn\,\frac{z}{\alpha}=\sn\left(\frac{z}{\alpha};k\right)
\]
maps the domain $F \backslash E$ onto the domain $\HH\backslash [\i \hat c, \i \hat d]$ where $\HH=\{w\;:\;\Im w>0\}$ is the upper half-plane, $\i \hat c=\sn(\i c/\alpha)$, $\i \hat d=\sn(\i d/\alpha)$ and $\hat c$, $\hat d$ are real numbers with $\hat d>\hat c>0$. Here, $\sn$ is the Jacobian elliptic sine function which is the inverse of the function~\cite[p.~218]{ky}
\[
\sn^{-1}\,w=\sn^{-1}(w;k)=\int_0^w\frac{dt}{\sqrt{(1-t^2)(1-k^2t^2)}}.
\]
The function $\sn^{-1}(w;k)$ is a Schwarz-Christoffel transformation mapping $\HH$ conformally onto the rectangle with corners
$\pm \K(k), \pm \K(k) + i \K(\sqrt{1-k^2})$ such that the points $-1/k, -1, 1, 1/k$ are mapped
to the corners.
Then, the M\"obius transformation 
\[
z \mapsto \frac{z-\i\hat c}{z+\i\hat c}
\]
maps the domain $\HH\backslash [\i \hat c, \i \hat d]$ onto the domain $\B^2 \backslash [0,(\hat d-\hat c)/(\hat d+\hat c)]$. Thus, by the invariance of conformal capacity,
\[
\capa(F,E) = \capa(\B^2,[0,(\hat d-\hat c)/(\hat d+\hat c)])
=\frac{2\pi}{\mu((\hat d-\hat c)/(\hat d+\hat c))}.
\]
The values of the capacity $\capa(F,E)$ for several values of $a,b,c,d$ are presented in Table~\ref{tab:Rec-slit}. The ring domain $F \backslash E$ can be regarded as a ring domain of the type considered in Subsection~\ref{sec:trapezium} with $m=6$ (see Figure~\ref{fig:rec-slit} (right)). Thus, a lower bound for the capacity $\capa(F,E)$ can be obtained from the inequality~\eqref{eq:cap-Qj} where the quadrilateral $Q_j$ are defined as in Subsection~\ref{sec:trapezium} (the vertices of the quadrilaterals are the black dots in Figure~\ref{fig:rec-slit} (right)). By symmetry, we have $\M(Q_1)=\M(Q_4)$,  $\M(Q_2)=\M(Q_3)$, and $\M(Q_5)=\M(Q_6)$. Using linear transformation, the quadrilateral $Q_1$ can be mapped to a quadrilateral in the upper half-plane such that the vertex $a$ is mapped to $0$ and the vertex $a+b\i$ is mapped to $1$. Then the value of $\M(Q_1)$ will be computed using the MATLAB function \verb|QM|. Similarly, the quadrilaterals $Q_2$ and $Q_5$ can be mapped onto quadrilaterals of the form described in Theorem~\ref{thm_quadriT} and then the MATLAB function \verb|QMt| is used to compute the values of $\M(Q_2)$ and $\M(Q_5)$. The values of the lower bound~\eqref{eq:LB-new} are given in Table~\ref{tab:Rec-slit}. 

\begin{table}[H]
 \caption{The values of the capacity $\capa(F,E)$ where $E$ and $F$ are given by~\eqref{eq:rec-seg} for several values of $a,b,c,d$.}\label{tab:Rec-slit}  
   \begin{tabular}{|l|l|l|l|l|l|} 
 \hline
 $a$ & $b$  & $c$  & $d$  & $\capa(F,E)$  & Lower bound~\eqref{eq:LB-new} \\ 
 \hline
 5  & 6   & 2    & 5    &4.17125447391152  & 4.03911136361855 \\  \hline
 5  & 6   & 1.5  & 4.5  &4.03909687993575  & 3.93625993555218 \\  \hline
 5  & 6   & 0.1  & 5.9  &11.0951405324743  & 10.9139325276210 \\  \hline
 5  & 6   & 2.95 & 3.05 &1.25467160179695  & 1.23114070623905 \\  \hline
 5  & 2   & 0.5  & 1.5  &4.0000011018024   & 3.64807484961254 \\  \hline
 10 & 1   & 0.25 & 0.75 &4.00013977481468  & 3.72908545066779 \\  \hline
 1  & 10  & 2.5  & 7.5  &11.7660734963185  & 5.84657534082154 \\  \hline
 1  & 4   & 1    & 3    &5.87687212650123  & 4.46146150262299 \\  \hline
\end{tabular}
\end{table}

\end{nonsec}

\section{Epilogue}\label{sec:epi}
Theorems \ref{KuhnauThm} and \ref{fwg71} are examples of results which, in terms of the perimeter of the set $E\subset\B^2$, quantify how the capacity $\capa(\B^2,E)$ depends on the ``size'' or the ``shape'' of the set $E$. We now discuss some known results of this type and thereby point out ideas for further studies.

For a compact set $E\subset\R^n$, we consider its \emph{tubular neighbourhood}, defined as
\begin{align}
E(t)=\{z\in\R^n\,|\,d(z,E)<t\}=\bigcup_{w\in E}B^2(w,t),    
\end{align}
and study  the function
\begin{align*}
c(E,t)\equiv\capa(E(t),E).   
\end{align*}
It is a well-known fact
that for every compact set $E \subset \R^n,$ the boundary of its tubular neighborhood
has a finite  $(n-1)$-dimensional measure. This fact was refined and further studied 
under various structure conditions on the set $E$ in \cite{KLV}.
Note that the function $c(E,t)$ here is  decreasing with respect to $t$ and its behaviour is closely related to the size of the set $E$ when $t\to0^+$. As we will see below, this dependency is mutual: if $E$ is ``thick'' or ``big'', there is a lower bound for $c(E,t)$ tending to $\infty$, whereas, if the function $c(E,t)$ converges to $\infty$ slowly, the set $E$ is small.

We say that the compact set $E$ is of \emph{capacity zero} if $c(E,t_0)=0$ for some $t_0>0$ and denote this by $\capa E=0\,,$ in the opposite case $\capa E>0\,$.

\begin{theorem} \emph{(J. V\"ais\"al\"a \cite{v2})}
If $\lim_{t\to0^+}c(E,t)<\infty$, then $E$ is of capacity zero.
\end{theorem}

The theorem above follows from the results in \cite{v2}. It should be observed that if
$E$ is of capacity zero, then $c(E,t)=0$ for all $t>0.$ In his PhD thesis \cite{hei}, V. Heikkala proved the next result and attributed the idea of its proof to J. M\'aly.

\begin{theorem} \emph{(V. Heikkala \cite[Thm 4.6]{hei})}
Let $h:(0,\infty)\to(0,\infty)$ be a decreasing homeomorphism, which satisfies $h(t)\to\infty$ as $t\to0^+$. Then there exists a compact set $E\subset\R^n$ with $\capa E>0$ satisfying $c(E,t)<h(t)$ for all $t\in(0,1)$.
\end{theorem}

Furthermore, Heikkala studied the function $c(E,t)$ in more detail under various measure theoretic thickness conditions. For instance, he proved a lower bound for $c(E,t)$ if $E$ is uniformly perfect and an upper bound if $E$ satisfies the so called Ahlfors condition. For these results, see~\cite{hei}. Heikkala's results were refined, extended and generalized by J. Lehrb\"ack \cite{Leh} to the context of metric measure spaces and more general capacities. 

The aforementioned results \cite{hei,Leh} dealing with uniformly perfect sets or sets satisfying the Ahlfors condition depend on the pertinent structure parameters of the set $E$ and hence so do the obtained growth estimates for $c(E,t)$. There are also results of other type where bounds such as
\begin{align}
c(E,t)\leq a_1t^{-n}c(E,1),\quad t\in(0,1),   
\end{align}
were proved for a compact set $E\subset\B^n$ with the constant $a_1$ only depending on $n$. See \cite[Lemma 8.22]{HKV}, \cite[Lemma 3.3, p. 60]{res}. The proof makes use of the Whitney extension theorem and standard gradient
estimates for mollifying functions, see also \cite{Leh} and \cite[Ch. 13]{m}. The point here is that the growth rate of $t^{-n}$ is independent of $E$ and the power  $-n$ is the best possible (independent of $E$) as shown
in \cite[p. 146]{HKV}.

We conclude  by discussing the possible use of the domain functional $d(E)\slash d(E,\partial G)$ in the estimation of the capacity $\capa(G,E)$ when $G\subset\R^2$ is a bounded simply connected domain. It follows from Lemma \ref{vu81} and \eqref{capvu81} that such a bound exists  if the set $E$ is connected.

The class of simply connected domains is too general for our purpose; it contains many potential theoretic counterexample domains such as ``rooms connected by narrow corridors'', which we would have to exclude. Thus we consider a suitable subclass of domains \cite[p. 84]{HKV}.

\begin{nonsec}{\bf $\varphi$-uniform domains.}
Let $\varphi:[0,\infty)\to[0,\infty)$ be an increasing homeomorphism. We say that a domain $G\subsetneq\R^n$ is $\varphi$-{\it uniform} if
\begin{align*}
k_G(x,y)\leq\varphi\left(\frac{|x-y|}{\min\{d_G(x),d_G(y)\}}\right)    
\end{align*}
holds for every $x,y\in G$.

Simple examples of $\varphi$-domains are convex domains, which satisfy the condition above with $\varphi(t)\equiv t$. More generally, suppose that there exists a constant $c\geq1$ such that, for all $x,y\in G$, there exists a curve $\gamma$ joining $x$ and $y$ so that $\ell(\gamma)\leq c|x-y|$ and $d_G(z)\geq(1\slash c)\min\{d_G(x),d_G(y)\}$  for all $z\in\gamma$. In this case, the domain $G$ is $\varphi$-uniform with $\varphi(t)\equiv c^2t$ \cite[2.19(2)]{vu81}.

Suppose now that $F$ is a compact set in a simply connected $\varphi$-uniform 
domain $G$. Then 
\begin{align*}
k_G(F)\leq\varphi\left(\frac{d(F)}{d(F,\partial G)}\right)     
\end{align*}
and, because clearly
\begin{align*}
r_{k-Jung}(F)\leq k_G(F),    
\end{align*}
the inequality~\eqref{capJung} yields 
\begin{align}
\capa(G,F)\leq\capa(G,B_k(\varphi(d(F)\slash d(F,\partial G)))).    
\end{align}
Next, recall that in a simply connected plane domain $G$ by \eqref{kvsrho}
\begin{align}
k_G(F)\leq 2\rho_G(F)    
\end{align}
and, finally, for a compact set $F$ in a $\varphi$-uniform simply connected planar domain $G$ we see by \eqref{hypdisk} that
\begin{align} \label{FinalRes}
\capa(G,F)\leq2\pi\slash\text{th}(U\slash2),\quad U=2\varphi(d(F)\slash d(F,\partial G)).    
\end{align}
\end{nonsec}

Further study of the connection between the domain functional 
$d(F)\slash d(F,\partial G) $ and $\capa(G,F)\,$ seems to be worthwhile. For instance,
sharp inequalities are unknown. 

\def\cprime{$'$} \def\cprime{$'$} \def\cprime{$'$}
\providecommand{\bysame}{\leavevmode\hbox to3em{\hrulefill}\thinspace}
\providecommand{\MR}{\relax\ifhmode\unskip\space\fi MR }
\providecommand{\MRhref}[2]{%
  \href{http://www.ams.org/mathscinet-getitem?mr=#1}{#2}
}
\providecommand{\href}[2]{#2}

\end{document}